\documentclass[11pt, a4paper]{article}
\usepackage{xcolor}

\usepackage{amsmath, amssymb}
\usepackage{hyperref}
\usepackage{algorithm}
\usepackage[noend]{algpseudocode}
\usepackage{bm}
\usepackage{flexisym}
\usepackage{amsthm}
\usepackage{adjustbox}

\newtheorem{theorem}{Theorem}
\newtheorem{lemma}{Lemma}
\newtheorem{proposition}{Proposition}
\newtheorem{conjecture}{Conjecture}
\newtheorem{observation}{Observation}

\usepackage[
  hyperref=auto,
  mincrossrefs=999,
  backend=biber,
  style=trad-abbrv, 
]{biblatex}
\begin{document}
\title{Generation of Cycle Permutation Graphs and Permutation Snarks}

\author{\sc 
Jan GOEDGEBEUR
    \footnote{Department of Computer Science, KU Leuven Campus Kulak-Kortrijk, 8500 Kortrijk, Belgium}\;\footnote{Department of Mathematics, Computer Science and Statistics, Ghent University, 9000 Ghent, Belgium}\;,
 Jarne RENDERS\footnotemark[1]\;\\\sc
and Steven VAN OVERBERGHE\footnotemark[2]\;
\footnote{E-mail addresses:
    \{jan.goedgebeur,
    jarne.renders\}@kuleuven.be, Steven.VanOverberghe@UGent.be} 
}

\date{}

\maketitle             
\begin{center}
\begin{minipage}{125mm}
\textbf{Abstract.}
We present an algorithm for the efficient generation of all pairwise non-isomorphic \emph{cycle permutation graphs}, i.e.\ cubic graphs with a $2$-factor consisting of two chordless cycles, non-hamiltonian cycle permutation graphs and \emph{permutation snarks}, i.e.\ cycle permutation graphs that do not admit a $3$-edge-colouring. This allows us to generate all cycle permutation graphs up to order $34$ and all permutation snarks up to order $46$, improving upon previous computational results by Brinkmann et al. Moreover, we give several improved lower bounds for interesting permutation snarks, such as for a smallest permutation snark of order $6\bmod 8$ or a smallest permutation snark of girth at least $6$ and give more evidence in support of a conjecture of Goddyn. These computational results also allow us to complete a characterisation of the orders for which non-hamiltonian cycle permutation graphs exist, answering an open question by Klee from 1972, and yield many more counterexamples to conjectures by Jackson and Zhang. 

\bigskip


\end{minipage}
\end{center}

\vspace{1cm}

\section{Introduction}

A \emph{cycle permutation graph} is a cubic graph containing a $2$-factor consisting of two chordless cycles. One can easily see that both of these cycles must have length $n/2$, where $n$ is the order of the graph. 

Cycle permutation graphs were first introduced by Chartrand and Harary~\cite{CH67} in 1967. They characterised the conditions under which a cycle permutation graph is planar.

In 1972, Klee~\cite{Kl72} also studied cycle permutation graphs, where he referred to them as ``generalised prisms''. He was interested in the hamiltonicity of cycle permutation graphs and asked for which orders there exist non-hamiltonian cycle permutation graphs. Using a computer search Klee determined that up to order $16$, the Petersen graph is the only non-hamiltonian cycle permutation graph and he partially solved the question by proving that for every order $n\equiv 2\bmod 4$ with $n\geq 18$ there exist non-hamiltonian cycle permutation graphs. In this manuscript, we finish this classification by completely characterising for which orders non-hamiltonian cycle permutation graphs exist. 

Cycle permutation graphs were also studied in the 80's by Shawe-Taylor and Pisanski, who asked questions about the girth of such graphs~\cite{PS82,PS81}. They construct cycle permutation graphs with an arbitrarily large girth.

We are also interested in cycle permutation graphs which are \emph{snarks}, i.e.\ non-$3$-edge colourable. In this case we call them \emph{permutation snarks}. We note that in some definitions, snarks are required to be \emph{cyclically $4$-edge-connected} and/or of \emph{girth} at least $5$,
however, this is always the case for permutation snarks of order $n > 6$. A graph is \emph{cyclically $k$-edge-connected} if the removal of fewer than $k$ edges cannot separate the graph into two components each containing a cycle and its \emph{girth} is the length of a shortest cycle. Note that a cycle permutation graph with a cycle of length $4$ is always hamiltonian and hamiltonian graphs cannot be snarks.
Moreover, an edge cut $X$ can only separate a cycle permutation graph into components if $X$ either contains all edges connecting the two chordless cycles or if $X$ contains at least two edges in one of the chordless cycles. If $\lvert X \rvert \le 3$, the former can only happen if the cycle permutation graph has order $6$ and the latter if one of the components is a single vertex.

Snarks are particularly interesting since for a lot of open conjectures it can be shown that if the conjecture is false, the smallest possible counterexamples are snarks. This has amongst others been proven for the Cycle Double Cover conjecture~\cite{Se79_2,Sz73} and Tutte’s 5-flow conjecture~\cite{Tu54}. 
A better understanding of the class of snarks can therefore lead to a better understanding of many long-standing open conjectures. Permutation snarks in particular are of interest as their additional structure makes them natural candidates for investigation and this structure in turn allows us to study this class computationally up to significantly higher orders.

It is well known that the Petersen graph is a permutation snark. In 1997, Zhang conjectured that it is the only permutation snark which is cyclically $5$-edge-connected~\cite{Zh97}. This conjecture was refuted by Brinkmann, Goedgebeur, H\"agglund and Markstr\"om~\cite{BGHM13} in 2013, when they developed a new generation algorithm for cubic graphs and snarks and generated all snarks up to 36 vertices
and determined that there are 12 cyclically $5$-edge-connected permutation snarks on $34$ vertices. They obtained counts for permutation snarks by using a filter approach, which given an input graph (given by the generation algorithm), determines whether or not it is a permutation snark. However, this method is not very efficient as in practice very few cubic graphs are cycle permutation graphs and even fewer are permutation snarks. For example on order $32$ only $0.24\%$ of the cubic graphs of girth at least~$4$ (the minimum girth of a cycle permutation graph of order $> 6$) are cycle permutation graphs. 
We improve the result by Brinkmann et al.\ by determining all permutation snarks up to order $46$, yielding many more counterexamples to Zhang's conjecture.

In 2019, Má\v{c}ajová and \v{S}koviera~\cite{MS19} gave three methods for constructing permutation snarks and used them to provide permutation snarks of cyclic connectivity $4$ and $5$ for sufficiently large order $n\equiv 2\pmod 8$. They mention that no permutation snark of order $n\equiv 6\pmod 8$ is known (an open problem posed by Brinkmann et al.~\cite{BGHM13}) 
and prove that such a smallest permutation snark -- if it exists -- must by cyclically $5$-edge-connected.  Note that cycle permutation graphs of order $n\equiv 0\pmod 4$ are trivially $3$-edge-colourable since the two cycles in the 2-factor have even length. 

We have developed two specialised algorithms
for the exhaustive generation of all pairwise non-isomorphic cycle permutation graphs of a given order. The first algorithm
is based on orderly generation~\cite{Fa78,Re78}. 
When interested in non-hamiltonian cycle permutation graphs or permutation snarks, we slightly adapt the algorithm in order to only take care of the most common isomorphisms.
This speeds up the algorithm significantly at the cost of outputting some isomorphic copies of the same graph.
We also describe a second algorithm 
using the canonical construction path method~\cite{Mc98}. Its implementation is slower than the approach using orderly generation, however, we use it as 
a verification
for the correctness of our implementation of the first algorithm.

The orderly generation approach is described in Section~\ref{sec:algorithm}, while the canonical construction path approach is described in Appendix~\ref{app:ccpm}. Their implementations can be found on GitHub~\cite{GRV24}. Variations of these algorithms allow to restrict the search to non-hamiltonian cycle permutation graphs, permutation snarks or graphs with a given lower bound on the girth in an efficient way. 
This allowed us to exhaustively generate all cycle permutation graphs and permutation snarks up to much higher orders.

The rest of this paper is organised as follows.
In Section~\ref{sec:algorithm}, we explain how the algorithms work and prove their correctness. 

In Section~\ref{sec:results}, we use the algorithms
to obtain counts for cycle permutation graphs as well as non-hamiltonian cycle permutation graphs and permutation snarks. For the latter, our algorithm allows us to generate all permutation snarks up to order $46$.
We also increase lower bounds for interesting permutation snarks such as the minimum order for a permutation snark of order $n\equiv 6\bmod 8$, which was also posed as an open problem in~\cite{BGHM13}.

In Section~\ref{sec:orders}, we finish the characterisation of the orders for which non-hamiltonian cycle permutation graphs exist, solving one of the questions by Klee~\cite{Kl72}.

In Section~\ref{sec:conjectures}, we give more counterexamples to conjectures by Zhang~\cite{Zh97} and Jackson~\cite{Ja93} and give computational evidence for a conjecture of Goddyn~\cite{GVM97}.

\section{Generation algorithms}\label{sec:algorithm}

In this section, we will present our algorithm for exhaustively generating all cycle permutation graphs of a given order $n$.
The basic idea is a backtracking approach, where we start from two cycles of length $n/2$ connected by one edge and add edges between the cycles in all possible ways. Clearly, this generates all cycle permutation graphs. However, such an approach will give a lot of isomorphic copies. One way to solve this is to keep track of all intermediate graphs and to stop adding edges whenever a graph isomorphic to a previously generated graph is found. However, this approach is very memory intensive. For example if one generates all cycle permutation graphs of order $28$ using this approach, approximately $160$ gigabytes of memory is needed. 
Therefore, we will use a different approach to avoid the generation of isomorphic copies.

In our algorithm (see Section~\ref{subsec:orderly}), we reject isomorphisms using a method based on orderly generation~\cite{Fa78,Re78}.
One labelling in each isomorphism class is determined to be (strongly) canonical and only these labellings will be accepted by the algorithm. Then every graph is generated precisely once without having to store any graphs in memory.
For efficiency's sake, in the case we want to generate non-hamiltonian cycle permutation graphs or permutation snarks,
we tolerate that some isomorphic graphs have multiple ``canonical'' labellings, which might lead to isomorphic graphs being output. See Sections~\ref{sec:non-ham}~and~\ref{sec:snarks}. We call these labellings \emph{weakly canonical}.
While this method can output isomorphic copies of the same graph, this approach is a lot faster as we can avoid performing certain expensive computations while still filtering the most common isomorphs. Moreover, since the number of graphs output in this way remains relatively small, it is feasible to obtain the exact number of graphs by filtering the isomorphic copies in a post-processing step.

We now describe these algorithms in more detail and prove their correctness. A description of the algorithm using the canonical construction path method can be found in Appendix~\ref{app:ccpm}.
Our implementation of the algorithms can be found on GitHub~\cite{GRV24} and in Appendix~\ref{app:correctness_tests} we describe how we extensively tested the correctness of the implementations in various ways.

\subsection{Orderly generation method}
\label{subsec:orderly}

The idea of this method is based on orderly generation, which was independently introduced by Farad\v zev~\cite{Fa78} 
and Read~\cite{Re78} in 1978, but had already been used by Rozenfel'd~\cite{Ro73} in 1973 for the generation of strongly regular graphs. Orderly generation defines a canonically labelled object 
for each isomorphism class and outputs only that one for each isomorphism class.
As we will see, the problem of generating cycle permutation graphs lends itself well to orderly generation and is much faster than, for example, the canonical construction path method, which we describe in Appendix~\ref{app:ccpm}. 

We first define some terminology. Let $G$ be a \emph{subcubic} graph of even order $n$, i.e.\ a graph in which all vertices have degree at most $3$. 
If $G$ contains two vertex-disjoint chordless cycles $C_1$ and $C_2$ of length $k:= n/2$, then $F := \{C_1, C_2\}$ is called a \emph{permutation $2$-factor} of $G$, i.e.\ a $2$-factor consisting of chordless cycles $C_1$ and $C_2$. We call any edges of $G$ that have one endpoint in $C_1$ and one endpoint in $C_2$ \emph{spokes} of $F$. If the vertices of degree $2$ in $G$ on one of the cycles induce a path, we call this cycle $\emph{consecutive}$ and we call $F$ a \emph{consecutive permutation $2$-factor} of $G$.  

We will represent subgraphs of cycle permutation graphs in the following way.
Let $G$ be a subcubic graph containing a permutation $2$-factor $F$ with cycles $C_1$ and $C_2$. 
Let $C_1 = u_0u_1\ldots u_{k-1}$ and $C_2 = v_0v_1\ldots v_{k-1}$. Let $[x]:= \{0,1,\ldots, x-1\}$ for the set of the first $x$ integers\footnote{Note that this might deviate slightly from how this notation is typically used.}. 
Then $G$ can be represented by the sequence $p = p_0, \ldots, p_{k-1}$ (also known as a partial permutation), where $p_i$ for $i\in [k]$ is given by
\[p: [k] \rightarrow [k] \cup \{{?}\}: i \mapsto
    \begin{cases}
    j, & \text{if $u_i$ is adjacent to $v_j$,}\\
    ?, & \text{otherwise.}
    \end{cases}
\]
We say $p$ \emph{represents} $G$ via $F$. Note that $G$ and $F$ are not sufficient to determine $p$, we also need an assignment of the vertices of $F$ to the variables $u_i, v_i$, with $i\in[k]$. Therefore, $G$ can have many such sequences representing it even via the same permutation $2$-factor $F$. 

On the other hand, a sequence $p$ gives rise to a labelled graph $G$. Label the vertices of $G$ from $0$ to $n-1$ such that if $i\in [k]$, then it has neighbours, $i-1$, $i+1$, taking indices modulo $k$, and $k + p(i)$, if $p(i)\neq {?}$. If $i\in [n]\setminus[k]$, then it has neighbours $i-1$, $i+1$, taking indices modulo $k$ and then adding $k$, and $j$ if there exists a $j\in [k]$ such that $p(j) = (i-k)$. In other words, if $p$ represents $G$ via $F$ with cycles $u_0\ldots u_{k-1}$ and $v_0\ldots v_{k-1}$, this labelling is obtained from the isomorphism mapping $u_i$ to $i$ and $v_i$ to $k+i$ for $i\in [k]$. We say this labelling, denoted by $G(p)$, is \emph{given by $p$}.

If $p$ represents $G$ via $F$, we call a sequence $p$ \emph{(strongly) canonical} if it is lexicographically smallest among all sequences representing $G$, where $i < {?}$ for all $i\in [k]$. We call it \emph{weakly canonical} if it is lexicographically smallest among all such sequences representing $G$ via $F$ (but as we will see later, being weakly canonical is independent of $F$). We emphasise that weak canonicity is a notion that is indeed weaker than (strong) canonicity by noting that the sequences representing $G$ via $F$ are a subset of all sequences representing $G$. We will use weak canonicity in Section~\ref{sec:non-ham}.

\begin{algorithm}[!htb]
    \caption{Expand($p$, $l$)}
    \label{alg:recursive_method_orderly}
    \begin{algorithmic}
        \If{$l = k$}
            \State Output cubic graph represented by $p$.
        \EndIf
        \ForAll{$x\in [k]\setminus \operatorname{im}(p)$}
            \State Define $p'$ by $p'(i) = p(i)$ for $i\in [k]\setminus\{l\}$ and $p'(l) = x$.
            \If{$p'$ is canonical} 
                \State Expand($p'$, $l+1$)
            \EndIf
        \EndFor 
    \end{algorithmic}
\end{algorithm}

A rough outline of the orderly generation algorithm's recursive method without any optimisations or technical details can be found in Algorithm~\ref{alg:recursive_method_orderly}, where we denote the image of $p$ by $\operatorname{im}(p)$. Algorithm~\ref{alg:recursive_method_orderly} outputs all non-isomorphic cycle permutation graphs of order $2k$ if started with $p = 0,{?}, {?}, \ldots, {?}$ and $l = 1$.

We now explain the algorithm in more detail. Suppose we want to generate the cycle permutation graphs of order $n$. Let $k:= n/2$. We start with $G_0$ consisting of two disjoint cycles $C_1 = u_0u_1\ldots u_{k-1}$ and $C_2=v_0v_1\ldots v_{k-1}$ and the edge $u_0v_0$. Note that this is a subgraph of any cycle permutation graph. $G_0$ is then represented by $p^1 = 0,{?}, {?}, \ldots, {?}$. Now let $G$ be a (sub)cubic supergraph of $G_0$ of order $n$ whose labelling is represented by $p^l$ with $(p^l)^{-1}([k]) = [l]$. If $G$ is cubic, then $k = l$ and we have a cycle permutation graph. It should be output and the recursion backtracks. If not, we recursively add edges from $u_l$ to $v_i$ with $i\in[k] \setminus \operatorname{im}(p)$ for all such possible $i$. Adding such an edge $u_lv_i$ gives us a new graph $G'$ whose labelling is represented by a sequence $p' = p_0, \ldots, p_{l-1}, i, {?}, \ldots, {?}$. We will now continue with $p'$ if it is canonical, i.e.\ lexicographically minimal among all sequences representing $G'$, or discard it otherwise. To clarify, we only output a cubic graph $G'$ if $p'$ is canonical. 

In this way we output exactly all non-isomorphic cycle permutation graphs of order $n$ precisely once. We describe how to determine canonicity of the sequences in Section~\ref{sss:canonicity}.

For what follows in Section~\ref{sss:finding} and for the proof of the algorithm's correctness in Section~\ref{sss:correctness}, we will use the following two results.

\begin{proposition}\label{prop:num2s}
    Let $G$ be a subcubic graph containing a permutation $2$-factor. Every permutation $2$-factor with cycles $C_1$ and $C_2$ will have the same number of degree $2$ vertices in $C_1$ and $C_2$ (namely half of all degree 2 vertices in $G$).
\end{proposition}
\begin{proposition}\label{prop:allconsecutive}
    Let $G$ be a subcubic graph containing a consecutive permutation $2$-factor with cycles $C_1$ and $C_2$. Then every permutation $2$-factor of $G$ will necessarily be consecutive and partition the degree $2$ vertices in the same way.
\end{proposition}
\begin{proof}
    Assume without loss of generality that $C_1$ is the cycle with consecutive degree 2 vertices (call this path of degree 2 vertices $P$). Then any cycle $C$ of a permutation $2$-factor containing a vertex of $P$ must contain all of $P$. But by Proposition~\ref{prop:num2s} it cannot contain any other degree 2 vertices. Therefore $C$ must also be consecutive.
\end{proof}

\subsubsection{Determining canonicity}\label{sss:canonicity}

We now describe how to determine (weak) canonicity of a given sequence $p$ representing a graph $G$ via $F$. We can generate every sequence representing $G$ via $F$ using three operations on $p$. 
We use the notation \emph{$\bmod\ k$} as a mapping sending an integer $i_1$ to an integer $i_2\in \{0,\ldots, k-1\}$ such that $i_2\equiv i_1\bmod k$.

The first operation is a rotation $T: p\mapsto q$ such that $q(i) = p((i+1)\bmod k)$. It corresponds to an isomorphism $\phi_T$ from $G(p)$ to $G(q)$ sending vertices $i\in [k]$ to $(i+1)\bmod k$ and the remaining vertices to themselves. This can be thought of as rotating the labels of the first cycle.
The second operation is a reflection $R: p\mapsto q$ such that $q(i) = p(k - 1 - i)$ for $i\in [k]$. It corresponds to an isomorphism $\phi_R$ from $G(p)$ to $G(q)$ sending vertices $i\in[k]$ to $k-1-i$ and the remaining vertices to themselves. This can be seen as reflecting the first cycle.  
The third operation is an inversion, $I: p\mapsto q$ such that $q(i) = p^{-1}(i)$ if $i\in \operatorname{im}(p)$ and ${?}$ otherwise. It corresponds to an isomorphism $\phi_I$ from $G(p)$ to $G(q)$ sending the vertices $i\in[k]$ to $k+i$ and vice versa. This can be seen as swapping the cycles. 

\begin{lemma}\label{lem:canonical}
    Let $G$ be a (sub)cubic graph containing a permutation $2$-factor.
    Then, any sequence representing $G$ via a permutation $2$-factor $F$ can be obtained from any other sequence representing $G$ via $F$ by applying some composition of the above three operations $T$, $R$ and $I$.
\end{lemma}
\begin{proof}
    Let $p$ be a sequence representing $G$ via $F$ and let $q$ be another such sequence. There exists an isomorphism $\phi$ between $G(p)$ and $G(q)$. We note that $\phi$ maps the cycles given by $0,\ldots, k-1$ and $k,\ldots, n-1$ to themselves or to each other. There are $2k$ ways a cycle can be mapped to another cycle, namely $k$ choices to map the first element and a choice of direction. Since we are mapping two cycles and might swap these, this yields $8k^2$ options. It follows that $\phi$ is some composition of $\phi_T$, $\phi_R$ and $\phi_I$. We note that a rotation or reflection of the second cycle can be obtained by swapping them first. Since these isomorphisms correspond to $T$, $R$ and $I$, the result follows.
\end{proof}
Therefore, in order to determine weak canonicity of a sequence $p$, it suffices to compare it to any other of the at most $8k^2$ sequences obtained by applying any composition of $T$, $R$ and $I$ to it. Hence, this does not depend on the given $2$-factor via which it represents $G$.

We can now check (strong) canonicity of $p$ as follows. For each permutation $2$-factor $F$ and a sequence $p^F$ representing $G$ via $F$, check if $p$ is lexicographically smaller than all rotations of $C_1$ and $C_2$ of $F$. Then check all rotations when the first cycle of $F$ is reflected ($R(p^F)$), the second is reflected ($I\circ R\circ I(p^F)$), the cycles are swapped ($I(p^F)$) or any combination of these. Once a sequence smaller than $p$ is found we can prune the search.
How we find all permutation $2$-factors is described in Section~\ref{sss:finding}.

We note that in practice, we use a slightly different notion of canonicity based on the difference lists of the sequences representing $G$ and we use many optimisations such as not checking sequences for which it can immediately be determined that they are larger. We also use optimisations that allow us to check many of these sequences in constant time. See Appendix~\ref{app:orderly_optimisations} for some details on these optimisations.

\subsubsection{\texorpdfstring{Finding all permutation {$\mathbf2$}-factors}{Finding all permutation 2-factors}}\label{sss:finding}

In order to know if a given sequence representing $G$ is canonical, we need to know all of $G$'s permutation $2$-factors. We keep track of this dynamically: we maintain a list of all the permutation $2$-factors in $G$. After adding an edge $e$, we only need to check for $2$-factors containing $e$. This already reduces the computational cost significantly.  However, one can often avoid the search for new $2$-factors altogether by checking the following condition.

\begin{proposition}\label{prop:new2factors}
    Let $G$ be a subcubic graph containing a consecutive permutation $2$-factor with cycles $C_1$ and $C_2$ such that $C_1$ is consecutive. Let $e = uv$ be a non-edge of $G$ with $u\in V(C_1)$ and $v\in V(C_2)$, such that $C_1$ is still consecutive in $G + e$.  If $v$ has a neighbour of degree $2$, any consecutive permutation $2$-factor of $G + e$ is also a consecutive permutation $2$-factor of $G$.
\end{proposition}
\begin{proof}
     Let $F$ be a consecutive permutation $2$-factor of $G$ with cycles $C_1$ and $C_2$. 
     By our choice of $e$, $F$ is also a consecutive permutation $2$-factor of $G+e$.
     By Proposition~\ref{prop:allconsecutive}, a new permutation $2$-factor $F'$, if it exists, will have two cycles $C'_1$ and $C'_2$ such that, without loss of generality, all degree $2$ vertices of $C_1$ belong to $C'_1$ and all degree $2$ vertices of $C_2$ belong to $C'_2$. 
     Note that if a degree 2 vertex lies on a cycle, both of its neighbours must lie on the same cycle.
     If $v$ has a neighbour of degree 2, then so must $u$ (by Proposition~\ref{prop:num2s}). But one of those neighbours must be in $C_1'$ and the other in $C_2'$, which means that $e$ cannot be part of a permutation 2-factor.
\end{proof}

If this condition holds, no new consecutive permutation $2$-factors will be introduced by adding $e$ to the graph. If the condition does not hold, we need to perform an exhaustive search for permutation $2$-factors containing the new edge. For this, we use a backtracking algorithm for finding an induced cycle of length $k$. If $G$ is not cubic, we initialise our cycle with the endpoints of the new edge and all vertices of the path, in which all interior vertices have degree $2$, between one of these endpoints and $u_0$. If $G$ is cubic, we can only initialise the cycle by the endpoints of the new edge. We then recursively extend this current path, by adding a neighbour of its end not introducing a chord. If this neighbour has another neighbour of degree $2$, then we need not add it as the current cycle we are constructing is already saturated with degree $2$ vertices, by Proposition~\ref{prop:num2s}. If the length of the path is less than $k$, but the ends do not have any good neighbours to add, we prune. If the length of the path is $k$ and we have a cycle, then we check whether its complement is also a cycle. If this holds, then both cycles are induced and we have found a new permutation $2$-factor. 

We remark that the complexity of this algorithm to find all permutation $2$-factors is in the worst case exponential in $n$, but computing the weakly canonical sequence is polynomial in $n$.

\subsubsection{Correctness of the algorithm}\label{sss:correctness}

We remark that even though orderly generation is a standard technique for isomorphism rejection, and even though it seems to be well-suited for the generation of cycle permutation graphs, it is actually not entirely trivial that (strong) orderly generation is correct here.
More specifically, if $F_1$ and $F_2$ are permutation $2$-factors of a subcubic graph $G$, then after adding more edges such that $F_1$ is still a permutation $2$-factor, $F_2$ could potentially have chords now.
Therefore, if we would decide that a sequence representing $G$ via $F_1$ is not canonical because a canonical sequence representing $G$ via $F_2$, we would miss graphs in the output.
It is only because we fill our initial $2$-factor in a consecutive way, and therefore Proposition~\ref{prop:allconsecutive} applies, that the technique works.

Let the restriction of $p$ to its first $l$ elements, denoted by $p\vert_l$, be defined by $p\vert_l(i) = p(i)$ if $i < l$ and $p\vert_l(i) = {?}$ otherwise.

\begin{theorem}\label{thm:orderly_correctness}
    Algorithm~\ref{alg:recursive_method_orderly} outputs every cycle permutation graph of a given order $n$ exactly once.
\end{theorem}
\begin{proof}
    Suppose that the algorithm outputs at least two distinct sequences $p$ and $p'$ representing a cycle permutation graph $G$. Then one is lexicographically smaller than the other, say $p$, which means that $p'$ is not canonical and hence would not have been output, a contradiction.

    Conversely, let $G$ be a cycle permutation graph of order $n$ and $k := n/2$. We show that it gets output by Algorithm~\ref{alg:recursive_method_orderly}. 
    Let $p^k$ be the lexicographically minimal sequence representing $G$ via any of its permutation 2-factors, say $F$.
    Now suppose for $m < k$ that $p^m$ is a canonical sequence such that $p^m = p^k\vert_m$.
    If we call Expand($p^m$, $m$), it is clear that the algorithm checks whether $p^{m+1}:= p^k\vert_{m+1}$
    is canonical. We now show it is canonical and that the algorithm calls Expand($p^{m+1}$, $m+1$) or terminates if $p^{m+1}$ represent a cubic graph.

    Assume $p:=p^{m+1}$ is not canonical, then there is some sequence $q$ representing the same graph, say $G^p$, via some $2$-factor $F'$ such that $q$ is lexicographically smaller than $p$. In particular, this means that $p(i) = q(i)\neq {?}$ for $i\in [l]$ for some $l < m+1$, but $p(l) > q(l)$. This means however that $q$ can be extended to a sequence $q^k$ representing $G$ via $F'$ such that $q^k(i) = q(i)$ for all $i$ for which $q(i) \neq {?}$. Indeed, because of Proposition~\ref{prop:allconsecutive}, $F'$ partitions the degree $2$ vertices of $G^p$ in the same way as $F$. Hence, all edges of $G$ not yet in $G^p$ must also be spokes of $F'$, since these edges are also spokes of $F$ and connect degree $2$ vertices of $G^p$. However, by construction, $q^k$ is now lexicographically smaller than $p^k$. This is a contradiction.
    Therefore, $p$ is canonical. By induction, the graph represented by $p^k$, i.e.\ $G$ will be output by the algorithm.
\end{proof}

The main bottleneck of this algorithm is dynamically updating the permutation $2$-factors of the graph under consideration. In order to drastically speed up the computations, one could consider only checking one permutation $2$-factor at the expense of generating some isomorphic copies for graphs having multiple permutation $2$-factors, but of course when one is interested in the exact counts, this is only feasible when the number of output graphs is relatively small and post-processing can be done in reasonable time and using reasonable memory
to filter the isomorphic copies and obtain the exact counts. 
In the case of non-hamiltonian cycle permutation graphs or permutation snarks the number of graphs generated was indeed relatively small, which is why we were able to apply the much faster approach that does not search for permutation $2$-factors.

This concludes the generation algorithm for (general) cycle permutation graphs. We end this section by noting that a small variation allows for the generation of cycle permutation graphs with a given lower bound on the girth, namely by, before expanding the graph with a new edge, checking whether this new edge does not introduce a small cycle with length less than the given lower bound on the girth. We have used this variation to generate cycle permutation graphs with girth at least $5$, $6$,\ldots, $9$ for the results in Section~\ref{sec:results}. 

Next, in Sections~\ref{sec:non-ham} and~\ref{sec:snarks} we describe what adaptations were made to generate only non-hamiltonanion cycle permutation graphs and permutation snarks, respectively.

\subsection{Non-hamiltonian cycle permutation graphs}\label{sec:non-ham}

Here we describe an adaptation of the previous algorithm in order to generate non-hamiltonian cycle permutation graphs. In Section~\ref{sec:snarks} we do the same for permutation snarks. Next to extra pruning conditions that are specific to non-hamiltonian graphs and snarks (which will be described later in this section), the main difference is that we use the previously mentioned weaker form of canonicity, namely, by only considering the starting permutation $2$-factor and not searching for any others. A rough outline for the recursive method of this approach (without any optimisations or technical details) is given by Algorithm~\ref{alg:recursive_method_weak_orderly}. It is the same as Algorithm~\ref{alg:recursive_method_orderly} but with the check \emph{if $p'$ is canonical} replaced by \emph{if $p'$ is weakly canonical}.

This approach, which we will denote as \emph{weak orderly generation}, will generate all cycle permutation graphs for the given order, but will include several isomorphic copies in the output. However, since the counts of non-hamiltonian cycle permutation graphs and permutation snarks remain relatively low for small orders, it is feasible for these variations to perform an isomorphism check afterwards (e.g.\ using \texttt{nauty}~\cite{MP14}) in order to determine the actual counts.

\begin{algorithm}[!htb]
    \caption{Expand\_Weak\_Orderly($p$, $l$)}
    \label{alg:recursive_method_weak_orderly}
    \begin{algorithmic}
        \If{$l = k$}
            \State Output cubic graph represented by $p$.
        \EndIf
        \ForAll{$x\in [k]\setminus \operatorname{im}(p)$}
            \State Define $p'$ by $p'(i) = p(i)$ for $i\in [k]\setminus\{l\}$ and $p'(l) = x$.
            \If{$p'$ is weakly canonical} 
                \State Expand\_Weak\_Orderly($p'$, $l+1$)
            \EndIf
        \EndFor 
    \end{algorithmic}
\end{algorithm}

A small adaptation to Algorithm~\ref{alg:recursive_method_weak_orderly} allows for the efficient generation of all non-hamiltonian cycle permutation graphs of a given order $n$. This is done by generating cycle permutation graphs of girth at least $5$ (recall that a cycle of length $4$ implies a hamiltonian cycle in these graphs) 
and by also rejecting an expansion if it leads to a hamiltonian graph. 
Note that because our algorithm works by adding edges, once a graph has a hamiltonian cycle, its descendants will also have one. 
We have adapted the program \texttt{cubhamg} included in \texttt{nauty}~\cite{MP14} to perform this hamiltonicity check as it is a very fast method for determining the hamiltonicity of (sub)cubic graphs. More details on this algorithm can be found in~\cite{BGM22}. We also apply lookaheads to avoid performing the hamiltonicity check as much as possible, as for this algorithm the hamiltonicity check is the main bottleneck.

In particular, we forbid three types of hamiltonian cycles. Given two cycles $C_1$ and $C_2$ of a permutation $2$-factor, a hamiltonian cycle has an even number of edges with an endpoint in $C_1$ and the other endpoint in $C_2$. We classify the hamiltonian cycles by how many of these edges they contain. Via lookaheads we forbid all hamiltonian cycles containing four and six of these edges. (Note that hamiltonian cycles containing two such edges are already dealt with, since they would introduce a cycle of length $4$.)

\begin{algorithm}[!htb]
    \caption{Expand\_Non-Hamiltonian($p$, $l$)}
    \label{alg:recursive_method_weak_orderly_non-ham}
    \begin{algorithmic}
        \If{$l = k$}
            \State Output cubic graph represented by $p$.
        \EndIf
        \ForAll{$x\in [k]\setminus \operatorname{im}(p)$}
            \State Define $p'$ by $p'(i) = p(i)$ for $i\in [k]\setminus\{l\}$ and $p'(l) = x$.
            \State Let $G$ be the graph represented by $p'$.
            \If{$G$ has $4$-cycle}
                \State \Return
            \EndIf
            \If{$G$ contains predetected hamiltonian cycle with $4$ or $6$ crossings}
                \State \Return
            \EndIf
            \If{$p'$ is weakly canonical and non-hamiltonian} 
                \State Predetect future hamiltonian cycles with $4$ or $6$ crossings
                \State Expand\_Non-Hamiltonian($p'$, $l+1$)
            \EndIf
        \EndFor 
    \end{algorithmic}
\end{algorithm}

For four such edges, one can look at the two previously added edges, and forbid a pair of edges from appearing consecutively. For example, suppose $C_1 = u_0u_1\ldots u_{k-1}$ and $C_2 = v_0v_1\ldots v_{k-1}$ are the two cycles of our permutation $2$-factor under consideration. If $u_xv_{x'}$ and $u_yv_{y'}$ are the last two edges which were added, we can often forbid $u_zv_{z'}$ and $u_{z+1}v_{(z+1)'}$ from being added consecutively when $v_{z'}$ and $v_{(z+1)'}$ are neighbours of $v_{x'}$ and $v_{y'}$, respectively and vice versa. However, we cannot forbid these edges when they do not form a hamiltonian cycle, but instead form two disjoint cycles spanning the whole graph. Whether this is the case can easily be checked by analysing the order of $v_{x'}, v_{y'}, v_{z'}$ and $v_{(z+1)'}$ on $C_2$. A similar method can be used to forbid edges which will lead to hamiltonian cycles with six edges between $C_1$ and $C_2$. 

The generation of non-hamiltonian cycle permutation graphs on orders $28$, $30$ and $32$ is faster by a factor of $2$ when using these lookaheads and seems to indicate a similar speedup for all orders.
Our attempts to apply the same lookaheads for a general number of such edges or simply for hamiltonian cycles with eight edges between $C_1$ and $C_2$ did not lead to a further speedup of our program. A rough outline of the recursive method can be found in Algorithm~\ref{alg:recursive_method_weak_orderly_non-ham}.

\subsection{Permutation snarks}\label{sec:snarks}
We will now explain the variation which allows us to generate permutation snarks. The main idea is to filter out graphs containing a $2$-factor in which all cycles have even length, i.e.\ \emph{an even $2$-factor}. If a cubic graph has such a $2$-factor, it is $3$-edge-colourable. We note that it is more efficient than generating non-hamiltonian cycle permutation graphs and filtering the snarks using post-processing (even though the latter is feasible for the orders we consider).

Similar ideas as in Section~\ref{sec:non-ham}, with extensive pruning in the recursive method of the algorithm, lead to an efficient algorithm for the generation of permutation snarks, however empirical tests show that a heuristic approach, with less extensive pruning (and which might allow some graphs which are not permutation snarks to be output, but which in practice never happens up to order at least $46$), is more efficient. 

Again, we generate cycle permutation graphs of girth at least $5$ (since a hamiltonian cycle is an even $2$-factor) and reject an expansion if it introduces an even $2$-factor in the graph. This is checked via backtracking in a way similar as  the lookaheads from Section~\ref{sec:non-ham} for detecting hamiltonian cycles.

If the given graph does not have an even $2$-factor, but the expansion with edge $e$ does, then $e$ must belong to the even $2$-factor. We build $2$-factors recursively. Let $F$ be the current permutation $2$-factor with cycles $C_1$ and $C_2$.  We consider spokes $e_1:= e$ and $e_1'$ whose endpoint on $C_2$ is adjacent to the endpoint of $e_1$ on $C_2$. (Note that there might be two options for $e_1'$ and we recursively check both.) Then $e_1$ and $e_1'$ will belong to the $2$-factor we are building and the edge on $C_2$ between the endpoints of $e_1$ and $e_1'$ will not. In the next step, we let $e_2'$ be a spoke such that its endpoint on $C_1$ is a neighbour of the endpoint of $e_1$ on $C_1$ and add a new pair $e_2$ and $e_2'$ to our $2$-factor as before, where $e_2$ is a spoke whose endpoint on $C_2$ is adjacent to the endpoint on $C_2$ of $e_2'$. The edge on $C_1$ between the endpoints of $e_1$ and $e_2'$ and the edge on $C_2$ between $e_2$ and $e'_2$ will not belong to the $2$-factor. We keep doing this recursively in all possible ways until the endpoint of $e_i$ on $C_1$ is adjacent to the endpoint of $e_1'$ on $C_1$. Then we have obtained a $2$-factor. If at any point we cannot find a spoke $e_i$ or $e_i'$, the search backtracks.

\begin{algorithm}[!htb]
    \caption{Expand\_Permutation\_Snark($p$, $l$)}
    \label{alg:recursive_method_weak_orderly_snark}
    \begin{algorithmic}
        \If{$l = k$}
            \State Output cubic graph represented by $p$.
        \EndIf
        \ForAll{$x\in [k]\setminus \operatorname{im}(p)$}
            \State Define $p'$ by $p'(i) = p(i)$ for $i\in [k]\setminus\{l\}$ and $p'(l) = x$.
            \State Let $G$ be the graph represented by $p'$.
            \If{$G$ has $4$-cycle}
                \State \Return
            \EndIf
            \If{$G$ contains predetected hamiltonian cycle with $4$ or $6$ crossings}
                \State \Return
            \EndIf
            \If{heuristic finds even $2$-factor of $G$}
                \State \Return
            \EndIf
            \If{$p'$ is weakly canonical} 
                \State Predetect future hamiltonian cycles with $4$ or $6$ crossings
                \State Expand\_Permutation\_Snark($p'$, $l+1$)
            \EndIf
        \EndFor 
    \end{algorithmic}
\end{algorithm}

Given a $2$-factor and all of the spokes at which it crosses from $C_1$ to $C_2$, we can determine in linear time in this number of spokes, whether or not it is even. If it is even, we do not have a snark and can prune the current expansion. If not, we continue the search for even $2$-factors. While this approach can find most even $2$-factors relatively quickly, it is not exhaustive if we backtrack whenever we find a $2$-factor. Whenever a $2$-factor is encountered but not even, we can continue the search with a new pair $e_i, e_i'$, which is not necessarily adjacent to another spoke of our $2$-factor, in all possible ways. However, empirical tests show that few extra even $2$-factors are found at the expense of a lot of extra computational work. Therefore, we backtrack whenever we find a $2$-factor. 
In practice all cycle permutation graphs up to order (at least)~$46$ obtained by running the algorithm using only this heuristic check for even $2$-factors turned out to be snarks.

Finally, we combine this heuristic for finding even $2$-factors with the predetection of hamiltonian cycles that cross four or six times which was used in Section~\ref{sec:non-ham}. Although these would have been detected by the heuristic for even $2$-factors, this small search often prunes faster than the depth-first heuristic. Adding this predetection speeds up the generation on orders 28--34 by a factor of approximately $1.7$ and on order $36$ by a factor of approximately $1.9$. (Note that while no permutation snarks exist on order $0\bmod 4$, because any permutation $2$-factor is even, 
we can still run through the search tree to get an idea of the algorithm's efficiency.) A rough outline of the recursive method for this approach can be found in Algorithm~\ref{alg:recursive_method_weak_orderly_snark}.

\section{Results}\label{sec:results}
In this section, we discuss the results of the computations that were performed using the implementations of the above mentioned algorithms. Our implementations are open source and can be found on GitHub~\cite{GRV24}. Next to proving the algorithm's correctness (cf.\ Theorem~\ref{thm:orderly_correctness}), we also performed various tests for verifying the correctness of the implementation. These are described in Appendix~\ref{app:correctness_tests}.

\subsection{Cycle permutation graphs}

\begin{table}[!htb]
\setlength\tabcolsep{1.8mm}
    \centering
    \begin{adjustbox}{max width=\textwidth}
        \begin{tabular}{c || r | r | r | r | r | r}
            Order   & $g\geq 4$                 & $g\geq 5$             & $g\geq 6$             & $g\geq 7$         & $g\geq 8$         & $g\geq 9$\\\hline
            $8$     & $2$                       & $0$                   & $0$                   & $0$               & $0$               & $0$\\
            $10$    & $4$                       & $1$                   & $0$                   & $0$               & $0$               & $0$\\
            $12$    & $10$                      & $1$                   & $0$                   & $0$               & $0$               & $0$\\
            $14$    & $28$                      & $3$                   & $0$                   & $0$               & $0$               & $0$\\
            $16$    & $123$                     & $11$                  & $1$                   & $0$               & $0$               & $0$\\
            $18$    & $667$                     & $59$                  & $0$                   & $0$               & $0$               & $0$\\
            $20$    & $4\,815$                  & $402$                 & $4$                   & $0$               & $0$               & $0$\\
            $22$    & $41\,369$                 & $3\,602$              & $9$                   & $0$               & $0$               & $0$\\
            $24$    & $411\,231$                & $37\,178$             & $84$                  & $0$               & $0$               & $0$\\
            $26$    & $4\,535\,796$             & $424\,252$            & $846$                 & $1$               & $0$               & $0$\\
            $28$    & $54\,828\,142$            & $5\,289\,603$         & $12\,597$             & $0$               & $0$               & $0$\\
            $30$    & $717\,967\,102$           & $71\,206\,645$        & $197\,921$            & $1$               & $0$               & $0$\\
            $32$    & $10\,118\,035\,593$       & $1\,027\,074\,710$    & $3\,334\,149$         & $6$               & $0$               & $0$\\
            $34$    & $152\,626\,831\,184$      & $15\,800\,380\,281$   & $58\,638\,599$        & $190$             & $0$               & $0$\\
            $36$    & ?                         & $258\,386\,596\,744$  & $1\,077\,159\,843$    & $4\,437$          & $1$               & $0$\\
            $38$    & ?                         & ?                     & $20\,642\,970\,164$   & $147\,820$        & $0$               & $0$\\
            $40$    & ?                         & ?                     & ?                     & $5\,166\,381$     & $0$               & $0$\\
            $42$    & ?                         & ?                     & ?                     & $167\,517\,630$   & $2$               & $0$\\
            $44$    & ?                         & ?                     & ?                     & $5\,265\,419\,873$& $33$              & $0$\\
            $46$    & ?                         & ?                     & ?                     & ?                 & $847$             & $0$\\
            $48$    & ?                         & ?                     & ?                     & ?                 & $21\,294$         & $0$\\
            $50$    & ?                         & ?                     & ?                     & ?                 & $1\,053\,289$     & $0$\\
            $52$    & ?                         & ?                     & ?                     & ?                 & $7\,281\,578$     & $0$\\
            $54$--$58$& ?                       & ?                     & ?                     & ?                 & ?                 & $0$\\
            $60$    & ?                         & ?                     & ?                     & ?                 & ?                 & $2$\\
            $62$    & ?                         & ?                     & ?                     & ?                 & ?                 & $61$\\
            $64$    & ?                         & ?                     & ?                     & ?                 & ?                 & $1\,654$\\
            $66$    & ?                         & ?                     & ?                     & ?                 & ?                 & $71\,213$\\
    
        \end{tabular}
    \end{adjustbox}
    \caption{We omit the single cycle permutation graph of girth 3 on order 6 from this table. The columns with $g\geq k$ indicate the counts of cycle permutation graphs with girth at least $k$ for each order.}
    \label{tab:counts_cycle_permutation_graphs}
\end{table}

We have used our algorithms to generate all cycle permutation graphs of a given order and also extended them to generate graphs with a lower bound on the girth. This can be done efficiently as the algorithm only adds edges, so we can prune as soon as we have a cycle which is smaller than the desired girth. The results obtained by our implementation of Algorithm~\ref{alg:recursive_method_orderly} can be found in Table~\ref{tab:counts_cycle_permutation_graphs}.
These graphs can also be obtained from the House of Graphs~\cite{CDG23} at \url{https://houseofgraphs.org/meta-directory/cubic} (up to the orders for which it was still feasible to store them). The runtimes which were needed to obtain these results can be found in Table~\ref{tab:runtimes_cycle_permutation_graphs} in Appendix~\ref{app:runtimes_general}. 

While the implementation of the weak orderly generation approach 
(Algorithm~\ref{alg:recursive_method_weak_orderly}) does not give exact counts, it is faster than Algorithm~\ref{alg:recursive_method_orderly} at giving an upper bound for the number of pairwise non-isomorphic cycle permutation graphs for a given order.
A comparison indicates that the factor with which the second algorithm is faster seems to grow as the order increases and  the ratio of non-isomorphic graphs versus all graphs output by this algorithm seems to increase as well. For example on order $16$, $96.85\%$ of the output graphs are non-isomorphic, while on order $34$, this is $99.21\%$. More details can be found in Table~\ref{tab:counts_orderly_generation} in Appendix~\ref{app:runtimes_general}.

\subsection{Non-hamiltonian cycle permutation graphs and permutation snarks}\label{subsec:nonham_cycle_perm_graphs}

\begin{table}[!htb]
    \centering
    \begin{tabular}{c || r | r | r || r | r | r | r }
        Order & $g\geq 5$ & $\chi' = 4$ & $\lambda_c \ge 5$ & $g\geq 6$ & $g\geq 7$ & $g\geq 8$ & $g\geq 9$\\\hline
        $10$    & $1$               & $1$       & $1$                   & $0$                   & $0$   & $0$    & $0$\\
        $12$--$16$    & $0$               & $0$       & $0$                   & $0$                   & $0$   & $0$    & $0$\\
        $18$    & $2$               & $2$       & $0$                   & $0$                   & $0$   & $0$    & $0$\\
        $20$    & $0$               & $0$       & $0$                   & $0$                   & $0$   & $0$    & $0$\\
        $22$    & $1$               & $0$       & $0$                   & $0$                   & $0$   & $0$    & $0$\\
        $24$    & $0$               & $0$       & $0$                   & $0$                   & $0$   & $0$    & $0$\\
        $26$    & $64$              & $64$      & $0$                   & $0$                   & $0$   & $0$    & $0$\\
        $28$    & $0$               & $0$       & $0$                   & $0$                   & $0$   & $0$    & $0$\\
        $30$    & $9$               & $0$       & $0$                   & $0$                   & $0$   & $0$    & $0$\\
        $32$    & $0$               & $0$       & $0$                   & $0$                   & $0$   & $0$    & $0$\\
        $34$    & $10\,778$         & $10\,771$ & $12$                  & $0$                   & $0$   & $0$    & $0$\\
        $36$    & $4$               & $0$       & $0$                   & $0$                   & $0$   & $0$    & $0$\\
        $38$    & $1\,848$          & $0$       & $0$                   & $0$                   & $0$   & $0$    & $0$\\
        $40$    & $19$              & $0$       & $0$                   & $0$                   & $0$   & $0$    & $0$\\
        $42$    & $3\,131\,740$     & $3\,128\,893$ & $736$             & $0$                   & $0$   & $0$    & $0$\\
        $44$    & $1\,428$          & $0$       & $0$                   & $0$                   & $0$   & $0$    & $0$\\
        $46$    & $678\,106$        & $0$       & $0$                   & $0$                   & $0$   & $0$    & $0$\\
        $48$    & ?                 & ?         & ?                     & $0$                   & $0$   & $0$    & $0$\\
        $50$--$54$ & ?              & ?         & ?                     & ?                     & $0$   & $0$    & $0$\\
        $56$--$60$ & ?              & ?         & ?                     & ?                     & ?     & $0$    & $0$\\
        $62$--$70$ & ?              & ?         & ?                     & ?                     & ?     & ?      & $0$\\

    \end{tabular}
    \caption{Counts of non-hamiltonian cycle permutation graphs for each order. Columns $g\geq k$ indicate counts with girth at least $k$ for each order, the column $\chi'=4$ indicates counts of permutation snarks and the column $\lambda_c\ge 5$ indicates counts of cyclically $5$-edge-connected permutation snarks. 
    }
    \label{tab:counts_nonham_cycle_permutation_graphs}
\end{table}

Certain adaptations to Algorithm~\ref{alg:recursive_method_weak_orderly} yield Algorithms~\ref{alg:recursive_method_weak_orderly_non-ham}~and~\ref{alg:recursive_method_weak_orderly_snark}, allowing for the efficient generation of all non-hamiltonian cycle permutation graphs or permutation snarks of a given order $n$, respectively.

Using these adaptations we obtain the counts of non-hamiltonian cycle permutation graphs and summarise them in Table~\ref{tab:counts_nonham_cycle_permutation_graphs}, and the following propositions, giving a partial answer to the questions asked by Klee~\cite{Kl72} concerning the hamiltonicity of cycle permutation graphs. We also increase the previously known counts of permutation snarks from order $34$ up to order $46$. The graphs are available in the meta-directory of the House of Graphs~\cite{CDG23} at \url{https://houseofgraphs.org/meta-directory/snarks}. Note that trivially there exist no permutation snarks on order $48$. 

While Algorithm~\ref{alg:recursive_method_weak_orderly} outputs isomorphic copies, we clarify that Table~\ref{tab:counts_nonham_cycle_permutation_graphs} contains the exact counts of non-hamiltonian cycle permutation graphs and permutation snarks as we filtered the isomorphic ones (using \texttt{nauty}~\cite{MP14}) whenever an adaptation of Algorithm~\ref{alg:recursive_method_weak_orderly} was applied. Since we also have an implementation of the canonical construction path method (see Algorithm~\ref{alg:recursive_method} in Appendix~\ref{app:ccpm}), we were able to independently verify the counts up to order $42$. Runtimes for the obtained counts using Algorithms~\ref{alg:recursive_method_weak_orderly_non-ham}~and~\ref{alg:recursive_method_weak_orderly_snark} and Algorithm~\ref{alg:recursive_method} are available in Table~\ref{tab:runtimes_nonham_cycle_permutation_graphs_orderly} and~\ref{tab:runtimes_nonham_cycle_permutation_graphs}, respectively, in Appendix~\ref{app:runtimes_nonham}. 

In contrast to the general case where the ratio of non-isomorphic graphs versus total graphs output by Algorithm~\ref{alg:recursive_method_weak_orderly} was close to $100\%$ and increasing, in the non-hamiltonian case, the ratio is a lot lower and can vary a lot depending on the order. For example, during the generation of non-hamiltonian cycle permutation graphs on orders $34$, $36$ and $38$, using Algorithm~\ref{alg:recursive_method_weak_orderly_non-ham}, we get $31.90\%$, $17.39\%$ and $36.07\%$, respectively. With the generation of permutation snarks (on orders where they exist), using Algorithm~\ref{alg:recursive_method_weak_orderly_snark}, the ratio seems to decrease. We have on orders $26$, $34$ and $42$, ratios $36.78\%$, $31.90\%$ and $29.04\%$, respectively. (See Table~\ref{tab:non-ham_counts_orderly_generation} in Appendix~\ref{app:runtimes_nonham} for more details.) An interesting observation is that non-hamiltonian cycle permutation graphs more often than not have multiple permutation $2$-factors, which is opposite from the general case. Since isomorphisms occur when output graphs have multiple permutation $2$-factors, this explains why there are many more isomorphic graphs output by the algorithm in this case.

It is easy to see that a permutation snark cannot have order $0\bmod 4$. Otherwise, the induced cycles of a permutation $2$-factor are even and we can colour their edges using two colours and give all other edges the third colour. Hence, a permutation snark can only exist for orders $2\bmod 4$. It was shown by Má\v{c}ájova and \v{S}koviera~\cite{MS19} that there exist (cyclically $5$-edge-connected) permutation snarks on every order $n\equiv 2\bmod 8$ for $n\geq 10$ ($n\geq 34$). However, so far no examples of order $n\equiv 6\bmod 8$ are known and in fact this was posed as an open problem in \cite[Problem~1]{BGHM13}, where all permutation snarks up to order $34$ were determined. Má\v{c}ájova and \v{S}koviera proved~\cite{MS19} that a smallest example, if it exists, must be cyclically $5$-edge-connected. From the permutation snarks, we filtered the ones which are cyclically $5$-edge-connected. These counts are found in the fourth column of Table~\ref{tab:counts_nonham_cycle_permutation_graphs}.

Using the results of Algorithm~\ref{alg:recursive_method_weak_orderly_snark} found in Table~\ref{tab:counts_nonham_cycle_permutation_graphs}, we showed the following.
\begin{proposition}
    A smallest permutation snark of order $6\bmod 8$ has order at least $54$.
\end{proposition}
We also obtained this in a second way, namely using Algorithm~\ref{alg:recursive_method_weak_orderly_non-ham}, we generated all non-hamiltonian cycle permutation graphs up to order $46$. See Table~\ref{tab:counts_nonham_cycle_permutation_graphs}. While there are $1\,848$ non-hamiltonian cycle permutation graphs of order $38$ and $879\,828$ such graphs of order $46$, using two independent algorithms, we verified that all of them are $3$-edge-colourable. 

Zhang's conjecture~\cite{Zh97} stating that the Petersen graph is the only cyclically $5$-edge-connected permutation snark was refuted by Brinkmann et al.~\cite{BGHM13}. Our search verified their $12$ counterexamples on order $34$ and found $736$ new ones on order $42$, cf.\ Table~\ref{tab:counts_nonham_cycle_permutation_graphs}. 

While Shawe-Taylor and Pisanski showed in~\cite{PS82} that the girth of cycle permutation graphs can be arbitrarily large, no permutation snarks of girth $6$ or higher are known. Our computational results from Table~\ref{tab:counts_nonham_cycle_permutation_graphs} imply the following. 

\begin{proposition}
    The smallest non-hamiltonian cycle permutation graph of girth at least $6$ has order at least $50$. The smallest such graph of girth at least $7$ has order at least $56$. The smallest such graph of girth at least $8$ has order at least $62$. The smallest such graph of girth at least $9$ has order at least $72$.
\end{proposition}

Note that as a corollary, we can replace ``non-hamiltonian cycle permutation graph'' in the above proposition with ``permutation snark''.

\section{Orders of non-hamiltonian cycle permutation graphs} \label{sec:orders}
In his 1972 paper~\cite{Kl72}, Klee studied two questions involving the non-hamiltonicity of cycle permutation graphs, attributed to Ralph Willoughby. The first one asks which permutation graphs admit a hamiltonian cycle. The second asks for which orders $n$ there exist non-hamiltonian cycle permutation graphs. Klee answered this second question partially by proving that for all $n\equiv 2\bmod 4$ there is a non-hamiltonian cycle permutation graph if and only if $n$ is neither $6$ nor $14$. We complete the characterisation of the orders for which non-hamiltonian cycle permutation graphs exist and thereby solve Klee's second question. 

To complete the characterisation, we will briefly reintroduce Klee's notation, but opt for $0$-indexing of the permutations. For more details on Klee's ideas we refer to his original paper~\cite{Kl72}.

One can define cycle permutation graphs of order $n=2k$ given a $k$-permutation. Let $G(\pi)$ be the cubic graph consisting of two disjoint $k$-cycles $x_0x_1\ldots x_{k-1}$ and $y_0y_1\ldots y_{k-1}$ and edges $x_iy_{\pi(i)}$ for $0\leq i\leq k-1$.
Let $G'(\pi)$ be the subgraph of $G(\pi)$ where we have removed the edges $x_{k-1}x_0$ and $y_{k-1}y_0$. 

An $E$-path in the graph $G'(\pi)$ is a path for which its ends are both in the set $E=\{x_0, x_{k-1}, y_0, y_{k-1}\}$. A hamiltonian $E$-pair is a pair of disjoint $E$-paths containing all vertices of $G'(\pi)$. The permutation $\pi$ is called \emph{good}
if it either has a hamiltonian $E$-path with endpoints $x_i$ and $y_j$ for $i,j\in\{0,k-1\}$ or if it has a hamiltonian $E$-pair such that one of the disjoint paths has endpoints $x_i$ and $y_j$ for $i,j\in\{0,k-1\}$ and the other has endpoints $x_{k-1-i}, y_{k-1-j}$. Otherwise, $\pi$ is \emph{bad}.

If $\pi$ is a $k$-permutation with $\pi(k-1) = k-1$, then we can naturally restrict $\pi$ from $[k]$ to $[k-1]$. We denote this permutation by $\overline{\pi}$, i.e.\ $\overline{\pi}(i) = \pi(i)$ for $i\in [k-1]$. 

Given $m$ $k_i$-permutations, for $0\leq i < m$, we can concatenate them to obtain a new $(\sum_{i=0}^{m-1} k_i)$-permutation $\pi:= (\pi_0,\ldots, \pi_{m-1})$, where $$\pi(i) = \sum_{j=0}^{l-1} k_j + \pi_i(i - \sum_{j=0}^{l-1} k_j), \text{ for } \sum_{j=0}^{l-1} k_j \le i < \sum_{j=0}^{l} k_j. $$

We will use the following propositions.

\begin{proposition}[Klee~\cite{Kl72}]\label{prop:Klee1}
    If $\pi$ is a $k$-permutation with $\pi(k)=k\geq 3$, then $G(\pi)$ is non-hamiltonian if and only if the $(k-1)$-permutation $\overline{\pi}$ is bad.
\end{proposition}

\begin{proposition}[Klee~\cite{Kl72}]\label{prop:Klee2}
    Suppose that $\pi_i$ is a $k_i$ permutation for $0\leq i < m$ and that the following three conditions hold:
    \begin{enumerate}
        \item for each $i$, either $k_i = 1$ or $\pi_i$ is bad;
        \item there is an even (possibly zero) number of $i$'s for which $k_i = 1$;
        \item $k_0\neq 1$, $k_{m-1}\neq 1$ and if $k_i = k_j = 1$ we have that $j\neq i+1$.
    \end{enumerate}
    Then the concatenation $\pi = (\pi_0,\ldots, \pi_{m-1})$ is bad.
\end{proposition}

We can now finish the characterisation of the orders.
\begin{proposition}
    There is a non-hamiltonian cycle permutation graph of order $n$ if and only if $n$ is even and $n \in \{10,18,22,26,30\}$ or $n\geq 34$.
\end{proposition}
\begin{proof}
    The orders $n\equiv 2\bmod 4$ were already characterised by Klee. By the counts of Table~\ref{tab:counts_nonham_cycle_permutation_graphs} we see that all cycle permutation graphs of order $n\equiv 0\bmod 4$ and $n\leq 32$ are hamiltonian. 

    Let $\pi_{10}$ be a bad $4$-permutation of the Petersen graph (the non-hamiltonian cycle permutation graph on order $10$), which exists due to Proposition~\ref{prop:Klee1}. Let $G_1$ be a non-hamiltonian cycle permutation graph on order $36$ and $\pi_{36}$ a bad $17$-permutation obtained from this graph, let $G_2$ be a cycle permutation graph on order $40$ and $\pi_{40}$ be a bad $19$-permutation of this graph and let $\pi_{1}$ be the trivial $1$-permutation. Such a $G_1$ and $G_2$ exist, see Table~\ref{tab:counts_nonham_cycle_permutation_graphs}.

    By Proposition~\ref{prop:Klee2}, we obtain that all permutations in the following sequence are bad. 
    $$\pi_{36}, \pi_{40}, (\pi_{36},\pi_{10}), (\pi_{40}, \pi_{10}), (\pi_{36},\underbrace{\pi_{10},\dots,\pi_{10}}_{k \text{ times}}), (\pi_{36},\pi_1,\pi_{10}, \pi_1, \underbrace{\pi_{10},\dots,\pi_{10}}_{l \geq 1\text{ times}})$$

    By Proposition~\ref{prop:Klee1}, these bad permutations imply the existence of a non-hamiltonian cycle permutation graph on order, respectively, $36, 40, 44, 48, 36+8k, 48+8l$, where $l\geq 1$, which finishes the proof.
\end{proof}

\section{Conjectures regarding permutation snarks}\label{sec:conjectures}

In~\cite{BGHM13}, Brinkmann et al.\ refuted Zhang's conjecture~\cite{Zh97} stating that the Petersen graph is the only cyclically $5$-edge-connected permutation snark, by finding $12$ counterexamples on order $34$. These examples turned out to also be counterexamples to other conjectures. (Cf.\ Refuted conjectures 5.17-5.24 in~\cite{BGHM13}.) So it seems worthwhile to use our new list of permutation snarks on 42 vertices to investigate these and related conjectures.

A \emph{cycle double cover} (usually abbreviated CDC) of a graph $G$ is multiset of cycles such that each edge lies in exactly two of the cycles.

The $12$ permutation snarks found by Brinkmann et al.\ are all known counterexamples to the following problem by Jackson, which was later posed as a conjecture by Zhang~\cite{Zh97}. 

\begin{conjecture}[Jackson \cite{Ja93}, Zhang \cite{Zh97}, refuted]\label{Conj:Jackson}
    Let $G$ be a cyclically $5$-edge-connected cubic graph and $\mathcal{D}$ be a set of pairwise disjoint cycles of $G$. Then $\mathcal{D}$ is a subset of a CDC, unless $G$ is the Petersen graph.
\end{conjecture}
Jackson~\cite{Ja93} notes that the conjecture is true in the case of $3$-edge-colourable cubic graphs and that a smallest counterexample must be cyclically $5$-edge-connected. 

The $12$ cyclically $5$-edge-connected permutation snarks on order {}$34$ are also counterexamples to the following weaker form of Jackson's conjecture, posed by Zhang~\cite{Zh97}.
\begin{conjecture}[Zhang \cite{Zh97}, refuted]\label{Conj:Zhang}
    Let $G$ be a cycle permutation graph with $C_1$ and $C_2$ the cycles of a permutation $2$-factor. If $G$ is cyclically $5$-edge-connected and there is no CDC which includes both $C_1$ and $C_2$, then $G$ must be the Petersen graph.
\end{conjecture}

We implemented a computer program for testing these two conjectures, which can be found in the same GitHub repository as our generation algorithm~\cite{GRV24}. 
Our computations resulted in the following observation. 
\begin{observation} \label{obs:conj_jackson}
    Out of the $736$ cyclically $5$-edge-connected permutation snarks of order $42$, there are exactly $448$ that have a permutation $2$-factor which cannot be part of any CDC and are thus counterexamples to Conjecture~\ref{Conj:Zhang} (and consequently also to Conjecture~\ref{Conj:Jackson}). In the remaining $228$ cyclically $5$-edge-connected permutation snarks on order $42$ any set of pairwise disjoint cycles is a subset of some CDC, so none are counterexamples to Conjecture~\ref{Conj:Jackson}.
\end{observation}
The set of counterexamples from Observation~\ref{obs:conj_jackson} as well as the set of non-counterexamples can also be found on GitHub~\cite{GRV24}.

A cycle $C$ in a cycle permutation graph $G$ is \emph{removable} for a permutation $2$-factor $F$ if $G - E(C)\cap E(F)$ is $2$-connected. In~\cite{GVM97}, the authors pose the following conjecture by Goddyn, which remains open.
\begin{conjecture}[Goddyn \cite{GVM97}]
    The only cycle permutation graph without removable cycles for any permutation $2$-factor is the Petersen graph. 
\end{conjecture}
They note that the conjecture holds true for $3$-edge-colourable cycle permutation graphs and cycle permutation graphs without a Petersen minor. Brinkmann et al.\ determined in~\cite{BGHM13} that the Petersen graph is the only permutation snark up to order $36$ with no removable cycles. They observed that on order $34$, the permutation snarks with the most number of removable cycles are the cyclically $5$-edge-connected ones. 

We wrote a program for counting the number of removable cycles in a cycle permutation graph and computed this for all permutation snarks up to order $42$. Its implementation can be found on GitHub~\cite{GRV24} as a feature of our filter program. 
The results of these computations are summarised in the following observation.

\begin{observation}
The Petersen graph is the only permutation snark on $n \le 48$ vertices with no removable cycles. The minimum number of removable cycles in a permutation graph on order $42$ is three. The 736 cyclically $5$-edge-connected permutation snarks of order 42 have between 80 and 205 removable cycles.
There exist 1\,121 permutation snarks on order $42$ which have at least $206$ removable cycles and are not cyclically $5$-edge-connected.
\end{observation}
This final observation is in contrast with the $34$-vertex case, where the cyclically $5$-edge-connected examples had the highest amount of removable cycles.

\section*{Acknowledgements}

The authors thank Gunnar Brinkmann, Edita Má\v cajová and Martin \v Skoviera for their valuable insights and contributions. 

This research of Jan Goedgebeur and Jarne Renders was supported by Internal Funds of KU Leuven and by an FWO grant with grant number G0AGX24N. Several of the computations for this work were carried out using the supercomputer infrastructure provided by the VSC (Flemish Supercomputer Center), funded by the Research Foundation Flanders (FWO) and the Flemish Government.

%
%
%

\printbibliography[title=\bibname]

\clearpage

\appendix

\section{Canonical construction path method}\label{app:ccpm}
In this section we describe how to use the canonical construction path method,
which is a generic method for the exhaustive isomorphism-free generation of graphs introduced by McKay in~\cite{Mc98},
in order to generate all pairwise non-isomorphic cycle permutation graphs.
This method (if used correctly) guarantees that every graph
is generated exactly once without having to store (and compare) all graphs in the memory
during the generation process.
Using similar ideas as in Section~\ref{sec:non-ham}, the algorithm can also be adapted to only generate non-hamiltonian permutation graphs.

An \emph{expansion} is an operation which constructs a larger graph from a given smaller one. The reverse operation is called a \emph{reduction}. To use the canonical construction path method we will need to define a \emph{canonical reduction} which is unique up to isomorphism and its inverse operation will then be the \emph{canonical expansion}. Since our expansions will consist of adding a non-edge to the graph, we say two expansions applied to the same graph $G$ are \emph{equivalent} if there exists an automorphism of $G$ mapping the respective non-edges to each other. This gives us, for each graph to which we apply expansions, classes of equivalent expansions. 

When applying an expansion we need to determine whether or not we will accept the newly generated graph and keep adding edges to it or if we discard it and choose another edge to add in the smaller graph. In order to avoid isomorphic copies, we should accept every non-isomorphic intermediate graph exactly once. This can be done by following these two rules: 
\begin{enumerate}
    \item \label{rule:CCPM_1} Only accept a graph if it was obtained by canonical expansion.
    \item \label{rule:CCPM_2} For every graph $G$ to which expansions will be applied, only perform one expansion from each equivalence class of expansions.
\end{enumerate}

In our case there is only one expansion operation, that is: adding an edge between pairs of specific vertices of degree $2$. 

Let $G$ be a spanning subgraph of a cycle permutation graph which has a consecutive permutation $2$-factor $F$ with induced cycles $C_1$ and $C_2$. For $i\in\{1,2\}$, if the degree $3$ vertices on $C_i$ induce a path $P$, let $S_i$ be the set of vertices of degree $2$ which are adjacent to $P$ on $C_i$ or let $S_i$ be empty otherwise. We consider a pair of vertices \emph{eligible} if one of the vertices lies in such a set $S_i$ and the other is a vertex of degree $2$ on $C_{3-i}$. We will apply the expansion to any eligible pair of vertices for any consecutive permutation $2$-factor of our graph.

Note that the algorithm would also work if we add an edge between any pair of degree $2$ vertices in which one lies on $C_1$ and the other on $C_2$. However, restricting the algorithm to eligible pairs drastically reduces the number of times the recursive algorithm branches and hence makes the implementation a lot more efficient.

The high level pseudocode for the recursive method of the algorithm can be found in Algorithm~\ref{alg:recursive_method}. 

\begin{algorithm}[!htb]
    \caption{Expand\_CCPM($G$)}
    \label{alg:recursive_method}
    \begin{algorithmic}
        \If{$G$ is cubic}
            \State Output $G$
        \EndIf
        \State Determine all eligible pairs of vertices.
        \State Determine the orbits of all such pairs. 
        \ForAll{eligible pairs $(u,v)$}
            \State // Cf.\ Rule~\ref{rule:CCPM_2} of the canonical construction path method.
            \If{$(u,v)$ is the representative of its orbit} 
                \If{adding edge $uv$ to $G$ is a canonical expansion} // Cf.\ Rule~\ref{rule:CCPM_1}.
                    \State Expand\_CCPM($G+uv$)
                \EndIf
            \EndIf
        \EndFor 
    \end{algorithmic}
\end{algorithm}

We now explain the algorithm in more detail. Suppose we want to generate the cycle permutation graphs of order $n$. We start with $G_0$ consisting of two cycles $C_1 = v_0v_1\ldots v_{n/2-1}$ and $C_2 = w_0w_1\ldots w_{n/2-1}$ and the edge $v_0w_0$.
Note that $C_1\cup C_2$ forms a consecutive permutation $2$-factor. Let $G$ be a (sub)cubic supergraph of $G_0$ of order $n$ which contains a consecutive permutation $2$-factor. 

If $G$ is cubic then it is a cycle permutation graph. It should be output and the recursion backtracks. If not, we determine all consecutive permutation $2$-factors in order to find the eligible pairs of vertices. In practice, we keep track of all consecutive permutation $2$-factors dynamically for efficiency reasons. More specifically: after adding an edge, we remove those which have become non-consecutive or are no longer permutation $2$-factors and search for consecutive permutation $2$-factors containing the newly added edge. While constructing the induced cycles of such a $2$-factor they are stored in a bitset, so that we can use bit operations to determine efficiently whether or not they have chords. This search for new consecutive permutation $2$-factors containing the most recently added edge is one of the bottlenecks of the algorithm.

Given all consecutive permutation $2$-factors, it is straightforward to determine all eligible pairs of vertices. We then need to determine the orbits of these pairs in order to only add an edge between one of the vertex pairs of each orbit, since all pairs in the same orbit would lead to isomorphic graphs. To obtain these orbits we use the program \texttt{nauty}~\cite{MP14} for determining all generators of the automorphism group of $G$. The orbits are then determined using a union-find algorithm. We choose the representative of each orbit to be the first vertex pair of that orbit we encountered while searching for eligible vertex pairs.

Once we have determined that an eligible vertex pair $(u,v)$ is the representative of its orbit, we need to determine whether its addition to $G$ will be a canonical expansion. In practice, we do this by looking at the graph $G+uv$ and analysing its \emph{reducible} edges. These are the edges $e$ for which $G + uv - e$ contains a consecutive permutation $2$-factor. These are however easily obtained once one knows all consecutive permutation $2$-factors of $G + uv$. Other edges are not reducible and can be ignored since removing them does not give us a graph which needs to be generated.

In order to determine whether or not we should accept $G + uv$ we need to define a canonical reduction which is unique up to isomorphism. The edge whose removal yields the canonical reduction will be the \emph{canonical edge}. To this end, we assign a $10$-tuple $(x_0, \ldots, x_9)$ to each reducible edge and let the canonical edge be the one with the lexicographically maximal value for this $10$-tuple.  

For a reducible edge $e=ab$, the values $x_0,\ldots, x_7$ are invariants of increasing discriminating power and cost, determined empirically. They are defined as follows: 
\begin{itemize}
    \item $x_0$ is the negative of the number of vertices at distance at most $2$ from $a$ or $b$. 
    \item $x_1$ ($x_2$) is the negative of the number of $4$-cycles ($5$-cycles) containing the edge $ab$.  
    \item $x_3$ is the number of vertices at distance at most $3$ from $a$ or $b$.
    \item $x_4$ is the negative of the number of $6$-cycles containing the edge $ab$. 
    \item $x_5$ ($x_6$) is (the negative of) the number of vertices of degree $2$ at distance $1$ (at most $2$) of $a$ and of $b$. 
    \item $x_7$ is the number of vertices at distance at most $4$ from $a$ or $b$. 
\end{itemize}
We note that while the discriminating power is negligible for the later invariants, they are necessary for the efficiency of variants of this program, such as the generation of non-hamiltonian cycle permutation graphs or cycle permutation graphs with girth restrictions. See Section~\ref{sec:results}. Moreover, their presence here does not noticeably increase the running time of the algorithm.

While the above values are invariant under isomorphism, their values could be the same for non-isomorphic graphs. Therefore we define $\{x_8, x_9\}$ to be the lexicographically largest label of an edge which is in the same edge orbit as $e$ in the canonical labelling of the graph. We again use \texttt{nauty}~\cite{MP14} for determining this canonical labelling of the graph. This gives us the generators of the automorphism group for free, hence if we accept this expansion, we take care not to call \texttt{nauty} a second time.

Note that in theory we could define the canonical edge using only $x_8$ and $x_9$, however, as calling \texttt{nauty} is computationally expensive, it is much more efficient to compute the other invariants first. Once such an invariant is able to show that our added edge $uv$ is not canonical, for example if there is a reducible edge with fewer vertices at distance at most $2$, then we can reject the expansion immediately and do not need to compute the remaining invariants (including the expensive call to \texttt{nauty}). Similarly, the invariants can be sufficient to determine if our added edge is the only reducible edge with maximum value for the tuple and hence it is canonical. The expansion can then immediately be accepted without computing the remaining invariants. As an example, on order $24$, invariants $x_0,\ldots, x_7$ are sufficient to discriminate the canonical edge in $94.64\%$ of the cases. If we were to only consider $x_8$ and $x_9$ for this order, this would slow down the total computation by a factor of $5$.

One of the bottlenecks of our algorithm is the search for new consecutive permutation $2$-factors containing the most recently added edge. This needs to happen after adding an eligible non-edge $e$ in order to determine all reducible edges and find out whether the addition of $e$ was a canonical expansion. We try to avoid this as much as possible. Suppose we are considering intermediate graph $G$ and eligible non-edge $e$. One of the ways we avoid this search is by first only removing the $2$-factors of $G$ which have become non-consecutive or which are no longer permutation $2$-factors in $G+e$ and determining a subset of the reducible edges based on this subset of consecutive permutation $2$-factors of $G+e$. One can then already compute the values $x_0,\ldots, x_9$ for this subset of reducible edges and one might be able to deduce whether $e$ is non-canonical. If this is the case, we can prune here.

If $e$ is not yet determined to be non-canonical then either it is canonical, or there is another consecutive permutation $2$-factor yielding new reducible edges, one of which is a canonical edge. Before starting the search for these $2$-factors, we first check the condition given by Proposition~\ref{prop:new2factors}.

If we find that this condition holds, then we have already computed all reducible edges in the previous step and determined whether or not $e$ is canonical. Otherwise, we will need to perform the algorithm for finding a consecutive permutation $2$-factor containing $e$. We can do this in the same way as was done in Algorithm~\ref{alg:recursive_method_orderly}.


\begin{theorem}\label{thm:ccpm_correctness}
    For a given order $n$, Algorithm~\ref{alg:recursive_method} will output all pairwise non-isomorphic cycle permutation graphs of order $n$ exactly once. 
\end{theorem}
\begin{proof}
    Let $G$ be a cycle permutation graph of order $n$. We show that it will be output by the algorithm. Let $G_m$ be a (sub)cubic subgraph of $G$ containing a consecutive permutation $2$-factor $F$ with $n+1 < m \leq 3n/2$ edges such that if it (or a graph isomorphic to it) is accepted by the algorithm, then $G$ (or a graph isomorphic to $G$) will be output. It is easy to see that if $G_{3n/2}$ is isomorphic to $G$ and gets accepted that (a graph isomorphic to) $G$ will be output by the algorithm. Since $G_m$ contains $F$, it also contains reducible edges and one of them must be the canonical edge, say $e$. Hence, if a graph isomorphic to $G_m - e$ was accepted by the algorithm, adding the image of $e$ under the isomorphism will be the canonical expansion, and a graph isomorphic to $G_m$ will be accepted. Finally, we complete the argument by noting that $G_{n+2} - e$, where $e$ is the canonical edge of $G_{n+2}$, is isomorphic to the starting graph of our algorithm, which shows that $G_{n+2}$ will be accepted.

    Conversely, we show that if $G$ is output by the algorithm it is the only graph of its isomorphism class which is output. Suppose that $G'$ is also output and is isomorphic to $G$. Let $$G_{n/2+1}, G_{n/2+2}, \ldots, G_{3n/2} = G$$ be the sequence of graphs with consecutive permutation $2$-factors where $G_m$ has $m$ edges and $G_{m-1} = G_{m} - e$ where $e$ is the canonical edge of $G_m$ which were accepted by the algorithm to eventually output $G$. Define $$G'_{n/2+1}, G'_{n/2+2}, \ldots, G'_{3n/2} = G'$$ similarly. Then either, for all $m\in \{n/2+1, \ldots, 3n/2\}$ we have that $G_m$ is isomorphic to $G'_m$. However, since $G_{n/2+1} = G'_{n/2+1}$ and there must be an $m'$ for which $G_{m'}\neq G'_{m'}$ as labelled graphs, we have performed an expansion using an eligible vertex pair which was not the representative of its orbit, which gives a contradiction. Hence, we have that there must be an $m'$ such that $G_{m}$ is isomorphic to $G'_{m}$ for all $m' < m \le 3n/2$, but with $G_{m'}$ and $G'_{m'}$ not isomorphic. Let $G_{m'+1} - e = G_{m'}$ and $G'_{m'+1} - e' = G'_{m'}$, then both $e$ and $e'$ must be canonical edges in their respective graphs and since $G_{m'+1}$ is isomorphic to $G'_{m'+1}$ there must be an isomorphism mapping these edges onto each other. However, this isomorphism then gives rise to an isomorphism between $G_{m'}$ and $G'_{m'}$, which is a contradiction. Hence, we conclude that the algorithm will only output one graph of each isomorphism class.
\end{proof}

\section{Correctness testing}\label{app:correctness_tests}

Next to proving the correctness of the algorithms, it is also important to verify the correctness of their implementations. To this end, we created a second implementation for generating all pairwise non-isomorphic (non-hamiltonian) cycle permutation graphs and permutation snarks based on the canonical construction path method of a given order. Its description can be found in Appendix~\ref{app:ccpm} and a rough outline of its recursive method is given in Algorithm~\ref{alg:recursive_method}. 

Firstly, we verified that the implementations of Algorithm~\ref{alg:recursive_method_orderly} and Algorithm~\ref{alg:recursive_method} yield the same results. All counts we checked in this way were in agreement. For cycle permutation graphs this was checked up to order $34$. For cycle permutation graphs of girth at least $5$ up to order $34$, of girth at least $6$ up to order $38$, of girth at least $7$ up to order $42$, of girth at least $8$ up to order $50$ and of girth at least $9$ up to order $64$. Variations of these algorithms allow us to generate non-hamiltonian cycle permutation graphs. We compared the results for non-hamiltonian cycle permutation graphs of girth at least $5$ up to order $42$, of girth at least $6$ up to order $42$, for girth at least $7$ up to order $50$, for girth at least $8$ up to order $56$, and for girth at least $9$ up to order $68$ and also here all counts were in agreement. 

Secondly, we wrote a filter program, which given a cubic graph as input, determines whether it has a permutation $2$-factor, by generating all perfect matchings of the graph, looking at their complement and verifying whether the corresponding 2-factor consists of two chordless cycles. Its implementation is open source and can be found in the same repository as our other programs on GitHub~\cite{GRV24}. We then combined this with a generator for cubic graphs in order to determine the counts of cycle permutation graphs. We used \texttt{snarkhunter}~\cite{BG17,BGM11}, which is a state-of-the-art generator for cubic graphs which allows to specify a lower bound on the girth. 

Using this generator-filter approach we verified the counts of cycle permutation graphs up to order $26$. For cycle permutation graphs with girth at least $5$, we verified this up to order $30$. For girth at least $6$, we verified up to order $34$ and for girth at least $7$ we verified up to order $38$. For cubic graphs of girth at least $8$ or $9$, we obtained these graphs from the meta-directory at the House of Graphs~\cite{CDG23} at \url{https://houseofgraphs.org/meta-directory/cubic}. We then used the filter on these graphs and verified the counts up to order $46$ and $64$, respectively. 

Another way we verified our counts is by using an algorithm that uses isomorphism-by-lists (i.e.\ the memory intensive approach mentioned in the first paragraph of Section~\ref{sec:algorithm}) instead of the orderly generation or canonical construction path method. Using this method, we also verified the counts of cycle permutation graphs up to order $28$, for girth at least $5$ also up to order $28$, for girth at least $6$ up to order $32$, for girth at least $7$ up to order $36$, for girth at least $8$ up to $46$ and for girth at least $9$ up to $62$. 

We have also used alternative implementations to verify the counts of non-hamiltonian cycle permutation graphs. One of them uses isomorphism-by-lists and a backtracking approach for checking hamiltonicity which uses the implementation that can be found on GitHub in~\cite{GRWZ22soft} and which was developed for the paper~\cite{GRWZ22}, instead of \texttt{cubhamg} (see Section~\ref{sec:non-ham}). Using this approach we verified the counts of non-hamiltonian cycle permutation graphs up to order $34$, with girth at least $6$ up to order $38$, with girth at least $7$ up to order $44$, with girth at least $8$ up to order $52$ and with girth at least $9$ up to order $64$. 

A second implementation uses the canonical construction path method as well as this method for checking the hamiltonicity of the intermediate graphs. With this implementation we verified the counts of non-hamiltonian cycle permutation graphs up to order $40$, with girth at least $6$ up to order $36$, with girth at least $7$ up to order $44$, with girth at least $8$ up to order $52$ and with girth at least $9$ up to order $64$. (Note that we allotted a lot more computation time to the generation of general non-hamiltonian cycle permutation graphs using this method, which is why we were able to verify this up to such a large order when compared to the generation of these graphs with girth restrictions.)

For all cases described above, the obtained results were in complete agreement, which gives a lot of confidence about the correctness of our implementation. Moreover, our implementations of these algorithms are open source software and can be found on GitHub~\cite{GRV24} where they can be verified and used by other researchers.

\clearpage

\section{Optimisations for determining canonicity}\label{app:orderly_optimisations}
When determining whether or not a sequence $p$ representing $G$ via a permutation $2$-factor $F$ is canonical during a run of Algorithm~\ref{alg:recursive_method_orderly}, we use several optimisations to speed up this process.

We first note that, in practice, we work with the difference lists of the sequences representing the graphs rather than the sequences themselves. While this slightly changes whether or not a labelling is canonical, the resulting graphs output by the algorithm will still be the same. 
This representation makes it easier to look at a sequence and see the effect of its rotations.

We define the difference list of a sequence $p$, where $k$ is the length of $p$, as 

\[d(p):[k]\rightarrow[k]\cup\{{?}\}: i\mapsto
\begin{cases}
   ? &\text{if $p(i) =\,?$ or $p(i+1) =\,?$,}\\
   p(i+1) - p(i)& \text{if $p(i) < p(i+1)$ or}\\
   k + p(i+1) - p(i)& \text{otherwise,}
\end{cases}
\]
taking indices modulo $k$.

Suppose we encounter $p$ in our algorithm and its first $m$ entries are filled, i.e.\ not equal to ${?}$, and the remaining entries are ${?}$. We use the knowledge that $p\vert_{m-1}$ was canonical to speed up some of the computations of determining whether $p$ is canonical. 

Let $d:= d(p)$ be the difference list of $p$ and let $F$ be the natural $2$-factor obtained from the labelling $G(p)$. When computing whether or not $d$ is canonical, we first check if it is weakly canonical before checking whether it is lexicographically smaller than sequences representing $G$ via a different permutation $2$-factor.
In order to check its weak canonicity, we need to check the rotations of eight lists, namely the $k$ rotations of $d$, the $k$ rotations of $d$ in the opposite direction, the $k$ rotations of the list where we subtract each element of $d$ from $k$, the $k$ rotations of taking both the opposite direction and subtracting and the $k$ rotations of the four previous lists, but for $d(I(p))$ instead of $d$. Note that $I(p)$ corresponds to swapping the cycles of $F$.

A naive approach would check in linear time for each of these lists whether or not $d$ is lexicographically smaller. However, for each of the lists we dynamically keep track which subsequences whose next element is ${?}$ are also a prefix of $d$. This allows us to often determine canonicity using only one comparison for each such subsequence. On average, there seem to be around 2 subsequences that match with $d$ distributed over the 8 lists, so in practice the canonicity test is almost free. We illustrate this with an example.

Let $p = 0,2,6,3,5,1,?,?$, then $d(p) = 2,4,5,2,4,?,?,?$. Suppose we want to compare $d(p)$ to all of its rotations. Since the final $4$ of $d(p)$ was a ${?}$ in the previous iteration (i.e.\ for $d(p\vert_{5})$), we store ``$2$'' as a subsequence which is a prefix of $d(p)$. We also store the empty sequence as a subsequence. Since the empty sequence has length $0$, we compare the first element of $d(p)$, i.e.\ $2$, to the new $4$. Since $2 < 4$ we cannot determine that $d(p)$ is not canonical.
Since the sequence ``$2$'' has length $1$, we compare the second element of $d(p)$, i.e.\ $4$ to the new $4$. If it were larger, then there is a rotation yielding a lexicographically smaller sequence than $d(p)$, if it were smaller, then this rotation is not lexicographically smaller and with the new element we cannot extend this subsequence to a new one which is still a prefix. Since in this case, the elements are equal and the $4$ is proceeded by a ${?}$, we extend the subsequence ``$2$'', to  ``$2,4$'' a new subsequence which is also a prefix of $d(p)$. We have now checked all rotations of $d(p)$ using only two comparisons. The same principle can be applied to the seven remaining lists we need to check. Using this optimisation to generate cycle permutation graphs using the weak orderly generation approach, the runtime decreases approximately by a factor of $3$ for orders $26$ and $28$, when compared to an approach that linearly determines the rotation of the list which is lexicographically minimal and pruning as soon as we can decide it will be smaller than $d(p)$. 

\section{Runtimes and counts for the generation of cycle permutation graphs}\label{app:runtimes_general}

\begin{table}[!htb]
    \centering
    \begin{tabular}{c || r | r | r | r | r | r}
        Order & $g\geq 4$ & $g\geq 5$ & $g\geq 6$ & $g\geq 7$ & $g\geq 8$ & $g\geq 9$ \\\hline
        $24$    & $1.89$s                   & $<1$s                 & $<1$s                 & $<1$s                   & $<1$s                   & $<1$s\\
        $26$    & $30.16$s                  & $3.98$s               & $<1$s                 & $<1$s                   & $<1$s                   & $<1$s\\
        $28$    & $533.48$h                 & $73.71s$s             & $<1$s                 & $<1$s                   & $<1$s                   & $<1$s\\
        $30$    & \bf$\bm{3.93}$h           & $1420.06$s            & $6.70$s               & $<1$s                   & $<1$s                   & $<1$s\\
        $32$    & \bf$\bm{82.86}$h          & \bf$\bm{10.89}$h      & $152.78$s             & $<1$s                   & $<1$s                   & $<1$s\\
        $34$    & \bf$\bm{0.22}$yrs         & \bf $\bm{248.10}$h    & \bf$\bm{1.41}$h       & $<1$s                   & $<1$s                   & $<1$s\\
        $36$    & /                         & \bf $\bm{0.68}$yrs    & \bf$\bm{36.59}$h      & $1.60$s                 & $<1$s                   & $<1$s\\
        $38$    & /                         & /                     & \bf$\bm{0.11}$yrs     & $60.29$s                & $<1$s                   & $<1$s\\
        $40$    & /                         & /                     & /                     & $2702.68$s              & $<1$s                   & $<1$s\\
        $42$    & /                         & /                     & /                     & \bf$\bm{40.19}$h        & $<1$s                   & $<1$s\\
        $44$    & /                         & /                     & /                     & \bf$\bm{0.19}$yrs       & $<1$s                   & $<1$s\\
        $46$    & /                         & /                     & /                     & /                       & $4.00$s                 & $<1$s\\
        $48$    & /                         & /                     & /                     & /                       & $257.71$s               & $<1$s\\
        $50$    & /                         & /                     & /                     & /                       & \bf$\bm{6.09}$h         & $<1$s\\
        $52$    & /                         & /                     & /                     & /                       & \bf$\bm{455.49}$h       & $<1$s\\
        $54$    & /                         & /                     & /                     & /                       & /                       & $<1$s\\
        $56$    & /                         & /                     & /                     & /                       & /                       & $<1$s\\
        $58$    & /                         & /                     & /                     & /                       & /                       & $<1$s\\
        $60$    & /                         & /                     & /                     & /                       & /                       & $1.57$s\\
        $62$    & /                         & /                     & /                     & /                       & /                       & $70.28$s\\
        $64$    & /                         & /                     & /                     & /                       & /                       & $1.30$h\\
        $66$    & /                         & /                     & /                     & /                       & /                       & \bf$\bm{203.94}$h\\

    \end{tabular}
    \caption{Runtimes for computing the entries in Table~\ref{tab:counts_cycle_permutation_graphs} using the (strong) orderly generation method. Bold entries were performed by parallelising the computation (which introduces some overhead) and executing it on the supercomputer of the Flemish supercomputer center (VSC), using Intel Xeon Platinum 8360Y (IceLake) CPUs.}
    \label{tab:runtimes_cycle_permutation_graphs}
\end{table}

\begin{table}[!htb]
    \centering
    \setlength\tabcolsep{1.1mm}
    \begin{tabular}{c||r| r | r || r | r | r}
        Order & \multicolumn{1}{c}{Alg.~\ref{alg:recursive_method_orderly} counts} & \multicolumn{1}{|c}{Alg.~\ref{alg:recursive_method_weak_orderly} counts} & \multicolumn{1}{|c||}{Rel. (\%)} & Time Alg.~\ref{alg:recursive_method_orderly} & Time Alg.~\ref{alg:recursive_method_weak_orderly} & Speedup\\\hline
        6 & 1 & 1 & 100.00 & $<1$s & $<1$s & / \\
        8 & 2 & 2 & 100.00 & $<1$s & $<1$s & /\\
        10 & 4 & 4 & 100.00 & $<1$s & $<1$s & /\\
        12 & 10 & 10 & 100.00 & $<1$s & $<1$s & /\\
        14 & 28 & 28 & 100.00 & $<1$s & $<1$s & /\\
        16 & 123 & 127 & 96.85 & $<1$s & $<1$s & /\\
        18 & 667 & 686 & 97.23 & $<1$s & $<1$s & /\\
        20 & 4\,815 & 4\,975 & 96.78 & $<1$s & $<1$s & /\\
        22 & 41\,369 & 42\,529 & 97.27 & $<1$s & $<1$s & /\\
        24 & 411\,231 & 420\,948 & 97.69 & 1.89s & 0.45s &                                  4.20\\
        26 & 4\,535\,796 & 4\,622\,509 & 98.12 & 30.16s & 4.79s &                           6.30\\
        28 & 54\,828\,142 & 55\,670\,332 & 98.49 & 533.48s & 62.07s &                       8.59\\
        30 & 717\,967\,102 & 726\,738\,971 & 98.79 & \bf3.93h & 861.05s &                   16.43\\
        32 & 10\,118\,035\,593 & 10\,217\,376\,792 & 99.03 & \bf82.86h & \bf5.02h &         16.50\\
        34 & 152\,626\,831\,184 & 153\,848\,448\,652 & 99.21 & \bf 0.22yrs & \bf95.86h &    19.81\\
        36 & ? & 2\,470\,073\,249\,960 & / & / & \bf0.15yrs &                        /\\
    \end{tabular}
    \caption{The number of cycle permutation graphs output by the implementations of Algorithms~\ref{alg:recursive_method_orderly}~and~\ref{alg:recursive_method_weak_orderly} using the strong and weak orderly generation methods (resp.\ second and third column). The fourth column displays the percentage of non-isomorphic graphs present in the third column. The fifth and sixth column display the running times to obtain these counts of Algorithm~\ref{alg:recursive_method_orderly} and \ref{alg:recursive_method_weak_orderly}, respectively. The final column displays by what factor the sixth column is faster than the fifth. Bold entries were performed by parallelising the computation (which introduces some overhead) and executing it on the Flemish supercomputer center (VSC), using Intel Xeon Platinum 8360Y (IceLake) CPUs.}
    \label{tab:counts_orderly_generation}
\end{table}

\clearpage
\section[Runtimes and counts for the generation of non-hamiltonian cycle permutation graphs and permutation snarks]{Runtimes and counts for the generation of\\ non-hamiltonian cycle permutation graphs and\\ permutation snarks}\label{app:runtimes_nonham}

\begin{table}[!htb]
    \centering
    \begin{tabular}{c || r | r || r | r | r | r}
        Order & $g\geq 5$  & $\chi' = 4$ & $g\geq 6$ & $g\geq 7$ & $g\geq 8$ & $g\geq 9$\\\hline
        $30$    & $3.08$s                  & $<1$s                 & $<1$s                 & $<1$s                 & $<1$s                 & $<1$s\\
        $32$    & $25.21$s                 & $4.43$s               & $<1$s                 & $<1$s                 & $<1$s                 & $<1$s\\
        $34$    & $233.91$s                & $40.11$s              & $3.98$s               & $<1$s                 & $<1$s                 & $<1$s\\
        $36$    & \bf$\bm{1.01}$h          & $275.93$s             & $35.34$s              & $<1$s                 & $<1$s                 & $<1$s\\
        $38$    & \bf$\bm{8.52}$h          & \bf $\bm{1.48}$h      & \bf$\bm{0.33}$h       & $<1$s                 & $<1$s                 & $<1$s\\
        $40$    & \bf$\bm{77.89}$h         & \bf $\bm{10.91}$h     & \bf$\bm{1.93}$h       & $2.83$s               & $<1$s                 & $<1$s\\
        $42$    & \bf$\bm{775.81}$h        & \bf $\bm{115.22}$h    & \bf$\bm{16.13}$h      & $28.49$s              & $<1$s                 & $<1$s\\
        $44$    & \bf $\bm{0.75}$yrs      & \bf $\bm{0.11}$yrs    & \bf$\bm{132.93}$h     & $316.74$s              & $<1$s                 & $<1$s\\
        $46$    & \bf $\bm{8.16}$yrs 
                                            & \bf $\bm{1.30}$yrs    & \bf $\bm{0.17}$yrs   & \bf$\bm{1.78}$h       & $<1$s                 & $<1$s\\
        $48$    & /                         & /                     & \bf $\bm{1.92}$yrs   & \bf$\bm{15.73}$h      & $1.16$s               & $<1$s\\
        $50$    & /                         & /                     & /                    & \bf $\bm{174.16}$h    & $14.83$s              & $<1$s\\
        $52$    & /                         & /                     & /                    & \bf $\bm{0.24}$yrs    & $236.57$s             & $<1$s\\
        $54$    & /                         & /                     & /                    & \bf $\bm{2.94}$yrs              & \bf $\bm{1.66}$h      & $<1$s\\
        $56$    & /                         & /                     & /                    & /                     & \bf $\bm{21.86}$h     & $<1$s\\
        $58$    & /                         & /                     & /                    & /                     & \bf $\bm{317.83}$h    & $<1$s\\
        $60$    & /                         & /                     & /                    & /                     & \bf $\bm{0.62}$yrs    & $0.73$s\\
        $62$    & /                         & /                     & /                    & /                     & /                     & $12.21$s\\
        $64$    & /                         & /                     & /                    & /                     & /                     & $258.18$s\\
        $66$    & /                         & /                     & /                    & /                     & /                     & \bf$\bm{2.81}$h\\
        $68$    & /                         & /                     & /                    & /                     & /                     & \bf $\bm{54.14}$h\\
        $70$    & /                         & /                     & /                    & /                     & /                     & \bf $\bm{0.12}$yrs\\

    \end{tabular}
    \caption{Runtimes for computing the entries in Table~\ref{tab:counts_nonham_cycle_permutation_graphs} using Algorithms~\ref{alg:recursive_method_weak_orderly_non-ham}~and~\ref{alg:recursive_method_weak_orderly_snark} based on weak orderly generation. Bold entries were performed by parallelising the computation (which introduces some overhead) and executing it on the Flemish supercomputer center (VSC), using Intel Xeon Platinum 8360Y (IceLake) CPUs.}
    \label{tab:runtimes_nonham_cycle_permutation_graphs_orderly}
\end{table}

\begin{table}[!htb]
    \centering
    \begin{tabular}{c || r | r | r | r | r}
        Order & $g\geq 5$ & $g\geq 6$ & $g\geq 7$ & $g\geq 8$ & $g\geq 9$\\\hline
        $26$    & $3.18$s                   & $<1$s                 & $<1$s                 & $<1$s                 & $<1$s\\
        $28$    & $26.76$s                  & $0.49$s               & $<1$s                 & $<1$s                 & $<1$s\\
        $30$    & $261.91$s                 & $4.59$s               & $<1$s                 & $<1$s                 & $<1$s\\
        $32$    & \bf$\bm{1.43}$h           & $43.99$s              & $<1$s                 & $<1$s                 & $<1$s\\
        $34$    & \bf$\bm{9.84}$h           & $429.50$s             & $<1$s                 & $<1$s                 & $<1$s\\
        $36$    & \bf$\bm{105.20}$h         & \bf$\bm{3.61}$h       & $1.71$s               & $<1$s                 & $<1$s\\
        $38$    & \bf$\bm{1062.77}$h        & \bf$\bm{26.53}$h      & $19.10$s              & $<1$s                 & $<1$s\\
        $40$    & \bf$\bm{1.41}$yrs         & \bf$\bm{251.97}$h     & $224.90$s             & $<1$s                 & $<1$s\\
        $42$    & \bf$\bm{17.66}$yrs        & \bf$\bm{3412.92}$h    & \bf$\bm{2.89}$h       & $<1$s                 & $<1$s\\
        $44$    & /                         & /                     & \bf$\bm{20.98}$h      & $0.50$s               & $<1$s\\
        $46$    & /                         & /                     & \bf$\bm{261.32}$h     & $6.02$s               & $<1$s\\
        $48$    & /                         & /                     & \bf$\bm{3041.08}$h    & $89.93$s              & $<1$s\\
        $50$    & /                         & /                     & \bf$\bm{5.24}$yrs     & $1480.43$s            & $<1$s\\
        $52$    & /                         & /                     & /                     & \bf$\bm{13.26}$h      & $<1$s\\
        $54$    & /                         & /                     & /                     & \bf $\bm{204.36}$h    & $<1$s\\
        $56$    & /                         & /                     & /                     & \bf $\bm{3448.13}$h   & $0.15$s\\
        $58$    & /                         & /                     & /                     & /                     & $2.17$s\\
        $60$    & /                         & /                     & /                     & /                     & $39.36$s\\
        $62$    & /                         & /                     & /                     & /                     & $838.34$s\\
        $64$    & /                         & /                     & /                     & /                     & \bf$\bm{11.60}$h\\
        $66$    & /                         & /                     & /                     & /                     & \bf$\bm{275.49}$h\\
        $68$    & /                         & /                     & /                     & /                     & \bf$\bm{6185.93}$h\\

    \end{tabular}
    \caption{Runtimes for computing the entries in Table~\ref{tab:counts_nonham_cycle_permutation_graphs} using Algorithm~\ref{alg:recursive_method} based on the canonical construction path method. Bold entries were performed by parallelising the computation (which introduces some overhead) and executing it on the Flemish supercomputer center (VSC), using Intel Xeon Platinum 8360Y (IceLake) CPUs. 
    }
    \label{tab:runtimes_nonham_cycle_permutation_graphs}
\end{table}

\begin{table}[!htb]
    \centering
	\begin{adjustbox}{max width=\textwidth}
        \begin{tabular}{c||c|c|c||c|c|c||}
            Order & {Non-ham.\ counts} & {Alg.~\ref{alg:recursive_method_weak_orderly_non-ham} counts} & Rel. (\%)  & {Snark counts} & {Alg.~\ref{alg:recursive_method_weak_orderly_snark} counts} & Rel. (\%)\\\hline
            10 & 1 & 1 & 100.00 & 1 & 1 & 100.00\\
            12-16 & 0 & 0 & 100.00 & 0 & 0 & 100.00\\
            18 & 2 & 4 & 50.00& 2 & 4 & 50.00\\
            20 & 0 & 0 & 100.00 & 0 & 0 & 100.00\\
            22 & 1 & 1 & 100.00& 0 & 0 & 100.00\\
            24 & 0 & 0 & 100.00 &0 & 0 & 100.00\\
            26 & 64 & 174 & 36.78 & 64 & 174 & 36.78\\
            28 & 0 & 0 & 100.00 & 0 & 0 & 100.00\\
            30 & 9 & 18 & 50.00 & 0 & 0 & 100.00\\
            32 & 0 & 0 & 100.00 & 0 & 0 & 100.00\\
            34 & 10\,778 & 33\,790 & 31.90& 10\,771 & 33\,767 & 31.90\\
            36 & 4 & 23 & 17.39 & 0 & 0 & 100.00\\
            38 & 1\,848 & 5\,124 & 36.07 & 0 & 0 & 100.00\\
            40 & 19 & 196 & 9.69 & 0 & 0 & 100.00\\
            42 & 3\,131\,740& 10\,789\,929 & 29.02 & 3\,128\,893 & 10\,774\,396 & 29.04\\
            44 & 1\,428 & 16\,640 & 8.58 & 0 & 0 & 100.00\\
            46 & 678\,106 & 2\,327\,096 & 29.14 & 0 & 0 & 100.00\\
        \end{tabular}
    \end{adjustbox}
    \caption{The number of non-hamiltonian cycle permutation graphs (permutation snarks) output by the implementation of Algorithm~\ref{alg:recursive_method_weak_orderly} using the weak orderly generation method. These counts are found in the third (sixth) column. The second (fifth) column shows the total number of pairwise non-isomorphic non-hamiltonian cycle permutation graphs (permutation snarks) for each order.}
    \label{tab:non-ham_counts_orderly_generation}
\end{table}

\end{document}